\documentclass[11pt]{amsart}
\usepackage{geometry}                
\usepackage{graphicx}
\usepackage{amsrefs}
\usepackage{amssymb}
\usepackage[matrix,arrow,ps]{xy}
\usepackage{epstopdf}
\newtheorem{theorem}{Theorem}[section]
\newtheorem{lemma}[theorem]{Lemma}
\newtheorem{cor}[theorem]{Corollary}
\newtheorem{prop}[theorem]{Proposition}

\theoremstyle{remark}
\newtheorem{remark}{Remark}
\theoremstyle{example}

\theoremstyle{definition}

\title{Curves homogeneous under analytic transformations}

\author{Giuseppe Della Sala}
\email{giuseppe.dellasala@univie.ac.at}\address{Fakult\"at f\"ur Mathematik \\ Universit\"at Wien}\address{Oskar-Morgenstern-Platz 1 \\ 1090 Wien \\ Austria}

\thanks{The author was supported by the START Prize Y377 of the Austrian Federal Ministry of Science and Research bmwf. The author was also partially supported by the Austrian Science Fund FWF grant P24878 N25.}

\begin{document}
\begin{abstract} We call a subset $K$ of $\mathbb C$ \emph{biholomorphically homogeneous} if for any two points $p,q\in K$ there exists a neighborhood $U$ of $p$ and a biholomorphism $\psi:U\to \psi(U)\subset \mathbb C$ such that $\psi(p)=q$ and $\psi(K\cap U)= K\cap \psi(U)$. We show that a biholomorphically homogeneous smooth curve $\gamma\subset \mathbb C$ is necessarily real-analytic.  
Furthermore we show that the same holds for the homogeneity with respect of a wider class of groups $G$ of real-analytic transformations of the plane. 
\end{abstract}

\maketitle

\section{Introduction}

Let $K$ be a subset of $\mathbb C$. We say that $K$ is \emph{locally biholomorphically homogeneous} if for any two points $p,q\in K$ there exist a neighborhood $U$ of $p$ in $\mathbb C$ and a biholomorphism $\psi:U\to \psi(U)\subset \mathbb C$, 
such that $\psi(U\cap K) = K\cap \psi(U)$ and $\psi(p) = q$.  One of the main results of the paper is the following:

\begin{theorem}\label{holhomo}
Let $\gamma\subset \mathbb C$ be a real $1$-dimensional curve of class $C^\infty$. Then $\gamma$ is locally biholomorphically homogeneous if and only if it is real-analytic.
\end{theorem}

The corresponding statement has been earlier proven in the setting of smooth diffeomorphisms, namely, the following result has been obtained. Recall that a set $K\subset \mathbb R^n$ is \emph{locally closed} if for every point $p\in K$ there exists a neighborhood $N$ of $p$ in $\mathbb R^n$ such that $N\cap K$ is a closed subset of $N$.
\begin{theorem}\label{Wilkin}
For any $r\in \mathbb N$, a locally closed subset of $\mathbb R^n$ is $C^r$-homogeneous if and only if it is a submanifold of class $C^r$.
\end{theorem}
The $C^1$ version of this result has been proved in \cite{RSS}, \cite{Sk}. In \cite{Wi}, this has been extended to the $C^r$ setting for any $r\in \mathbb N$. In section \ref{lcs}, we propose an alternative way of deriving Theorem \ref{Wilkin} from the case $r=1$ .

Although very natural, the real-analytic version of the homogeneity problem appears to be still open in its full generality, see the remarks at the end of \cite{Sk}. Combining Theorem \ref{Wilkin} with Theorem \ref{holhomo}, one gets that any locally closed subset of $\mathbb C$ which is homogeneous under holomorphic diffeomorphisms is either discrete, an open set or a real-analytic curve (see Th. \ref{disopcur}). 

\

\subsection*{Informal outline of the proof}

\

\

In order to obtain Theorem \ref{holhomo}, our general idea is the following: if we could show that $\gamma$ is an integral curve of a non-vanishing holomorphic vector field, then it would immediately follow that it is real-analytic. Suppose that $0\in \gamma$, and for any $p\in \gamma$ let $\psi_p:U\to \psi_p(U)$ be a local biholomorphism preserving $\gamma$ such that $\psi_p(0)=p$. We try to construct a vector field by a suitable limit process involving the maps $\psi_p$, using very classical methods developed by Cartan \cite{Ca}.

To be able to use these methods, first of all we need to show that the $\psi_p$ (when appropriately selected) are all defined on a common domain and are bounded uniformly with respect to $p$; this is accomplished in Lemma \ref{Baire} by means of a Baire categorical argument. Even in this situation the behavior of the $\psi_p$ is not yet enough under control to apply the results in \cite{Ca}. We can, however, apply them if we know additionally that the family $\{\psi_p\}$ is a group, that is for any $p_1,p_2\in \gamma$ the composition $\psi_{p_1}\circ\psi_{p_2}$ coincides with $\psi_{p_3}$ where $p_3=\psi_1(\psi_2(0))$. (It turns out that the existence of an inverse is easier to ensure).

In general there is no guarantee that the group property is satisfied. We try to overcome this obstacle in the following way: even though $\psi_{p_1}\circ\psi_{p_2}$ and $\psi_{p_3}$ need not coincide, they must always differ by an \emph{isotropy} of $\gamma$, that is $\psi_{p_1}\circ\psi_{p_2}=  \psi_{p_3}\circ\phi$ for some local biholomorphism $\phi$ fixing $0$ and preserving $\gamma$. If we can show that a non-analytic curve $\gamma$ does not admit any non-trivial isotropy (if $\gamma$ is real-analytic we are already done), then the group property holds and the methods in \cite{Ca} can be used.

Unfortunately it is not quite true that a non-analytic curve cannot have any holomorphic isotropy. We can show that a weaker property (see condition (ICP) introduced before Lemma \ref{onlyone}) holds instead: unless $\gamma$ is real-analytic outside of $0$, the only isotropies it can admit are involutions. The proof of  this property uses the Leau-Fatou flower theorem describing the dynamics of one-dimensional biholomorphism germs, and is given in Lemma \ref{leaufatou}. As it turns out, property (ICP) is nevertheless good enough: as shown in Proposition \ref{Ghom} it can exploited to construct a non-vanishing vector field and thus to conclude the proof of Theorem \ref{holhomo}.

\

\subsection*{Plan of the proof}

\

\

 Theorem \ref{holhomo} will be proved as a consequence of the following results: to state them we first need to give some definitions.

Denote by ${\rm Diff}^\omega(\mathbb R^2,0)$ the group of germs at $0$ of real-analytic diffeomorphisms defined in a neighborhood of $0\in \mathbb R^2$, equipped with the direct limit topology, that is, given $\{h_j\}_{j\in \mathbb N}\subset {\rm Diff}^\omega(\mathbb R^2,0)$ and $h\in {\rm Diff}^\omega(\mathbb R^2,0)$ we have $h_j\to h$ if and only if there is a neighborhood $U$ of $0\in \mathbb R^2$ and analytic diffeomorphisms $\tilde h_j:U\to \tilde h_j(U)\subset \mathbb R^2$, $\tilde h:U\to \tilde h(U)\subset \mathbb R^2$ such that the germ induced by $\tilde h_j$ (resp. $\tilde h$) at $0$ is  $h_j$ (resp. $h$) and $h_j\to h$ uniformly on $U$. Let $\mathcal G$ be a subgroup. We say that $\mathcal G$ is \emph{closed} if it is a closed subset of ${\rm Diff}^\omega(\mathbb R^2,0)$. For any $q\in \mathbb R^2$, denote by $\tau_q:\mathbb R^2\to\mathbb R^2$ the translation $\tau_q(x) = x + q$. We will say that a subgroup $\mathcal G\leq {\rm Diff}^\omega(\mathbb R^2,0)$ is \emph{uniform} if the following condition is satisfied: for any domain $U\subset \mathbb R^n$, $0\in U$, any $\psi \in {\rm Diff}^\omega(U)$ such that the germ of $\psi$ at $0$ belongs to $\mathcal G$, and for all $q\in U$ the germ at $0$ of the map $ \tau_{-\psi(q)} \circ \psi \circ \tau_{q}$ also belongs to $\mathcal G$. In other words, we require the representatives of the germs of $\mathcal G$ to induce a germ of $\mathcal G$ at any point of their domain (after suitable translations). 
 It is immediate to check that the subgroup ${\rm Hol}(\mathbb C,0)$ of holomorphic germs is both closed and uniform. We say that $\mathcal G$ satisfies \emph{property (ICP)} if for any non-involutive $g\in \mathcal G$ and any curve $\lambda\subset \mathbb R^2$ such that $0\in \lambda$ and $g(\lambda)\subset \lambda$ we have that $\lambda\setminus \{0\}$ is real-analytic (that is, for any $p\in \lambda$, $p\neq 0$ there is a neighborhood $U$ of $p$ such that $U\cap \lambda$ is real-analytic). Finally, we say that a curve $\gamma$ is \emph{$\mathcal G$-homogeneous} if for any two points $p,q\in \gamma$ there exist a neighborhood $U$ of $p$ in $\mathbb \mathbb R^2$ and a real-analytic diffeomorphism $\psi:U\to \psi(U)\subset \mathbb R^2$ such that $\psi(U\cap \gamma) = \gamma\cap \psi(U)$, $\psi(p) = q$ and the germ of $\psi$ at $p$ belongs to $\mathcal G$ (cf. the definition in section \ref{locequivpro}).

The most general result we obtain can now be formulated as follows: 

\begin{prop}\label{Ghom} Let $\mathcal G\leq {\rm Diff}^\omega(\mathbb R^2,0)$ be a closed uniform subgroup satisfying (ICP). Then every $\mathcal G$-homogeneous smooth curve $\gamma\subset \mathbb R^2$ is real-analytic.
\end{prop}

The second result we need is the following

\begin{lemma} \label{leaufatou}
The group ${\rm Hol}(\mathbb C,0)$ satisfies (ICP). Indeed it satisfies the following stronger property: let $f\in {\rm Hol}(\mathbb C,0)$, $f$ not an involution, and let $\gamma$ be a curve of class $C^1$, defined in a small neighborhood of $0$ in $\mathbb C$ and with $0\in \gamma$. Suppose that $\gamma$ is invariant under $f$. Then $\gamma\setminus\{0\}$ is real-analytic.
\end{lemma}

The proof of this statement is relatively elementary (relying on the theory of complex one-dimensional dynamical systems) but we could not find an explicit reference in the literature. Clearly, combining this lemma with Proposition \ref{Ghom} we obtain Theorem \ref{holhomo}.

The proof of Proposition \ref{Ghom} is a consequence of two claims:
\begin{itemize}
\item[\bf Claim 1] Let $G$ be a locally closed group of holomorphic transformations, seen from the local point of view, such that the identity transformation $Id$ is not isolated. Then $G$ admits a non-trivial infinitesimal transformation.
\item[\bf Claim 2] If $\gamma$ satisfies the assumptions of Proposition \ref{Ghom}, then either $\gamma$ is real-analytic or there is a subset $\{\psi_p\}$ of the local diffeomorphisms leaving $\gamma$ invariant which form a locally closed group of holomorphic transformations, seen from the local point of view, in which $Id$ is not isolated.
\end{itemize}
The definition of locally closed group of holomorphic transformations seen from the local point of view is borrowed from the work of Cartan \cite{Ca} and is too long to be stated here: we refer to section \ref{Cartan} for the details. The  statement that $G$ admits a non-trivial infinitesimal transformation means that there exists a non-vanishing analytic vector field whose flow consists of elements of $G$. Putting together Claim 1 and Claim 2 we show that either $\gamma$ is real-analytic or it is an integral curve of a non-vanishing analytic vector field, and thus in any case real-analytic.

Claim 1 coincides with Corollary \ref{puttog}, and its proof is obtained in section \ref{Cartan} as direct consequence of the results in \cite{Ca}. The hard part in the proof of Proposition \ref{Ghom} is to establish Claim 2. To achieve this, the two main ingredients are the following:
\begin{itemize}
\item[\bf Claim 3] If $\gamma$ satisfies the assumptions of Proposition \ref{Ghom} there is a subset $\{\psi_p\}$ of local diffeomorphisms leaving $\gamma$ invariant which are defined on a common domain and are uniformly bounded.
\item[\bf Claim 4] Let $\mathcal G$ be a group satisfying (ICP), and let $\gamma$ be a nowhere analytic curve of class $C^\infty$, $0\in \gamma$. Then there exists at most one (involutive) non-trivial element $g\in \mathcal G$  such that $\gamma$ is $g$-invariant.
\end{itemize}

Claim 3 is a particular case of the more general Lemma \ref{Baire}, which also contains further technical details that are needed later in the proof. The reason we need this claim is that the uniformity of the domain is part of the definition of group of holomorphic transformations seen from the local point of view given in section \ref{Cartan}.

Claim 4 is a restatement of Lemma \ref{onlyone}, and it is through this claim that property (ICP) enters in the proof of Proposition \ref{Ghom}. The Lemma is used in the following way: if the homogeneous curve $\gamma$ is not real-analytic, then it is necessarily nowhere analytic. Because of the claim, its isotropy group at any point consists of at most two elements. This turns out to be enough to select a subset $\{\psi_p\}$ of automorphisms of $\gamma$ forming a locally closed group of holomorphic transformations, seen from the local point of view, thus establishing Claim 2.

\

\subsection*{Motivations and generalizations}

\

\

\noindent In the proof of Theorem \ref{holhomo} outlined above, the holomorphicity of the diffeomorphisms involved is used only in the study of the dynamics of isotropies. This analysis can be performed in other cases, too, allowing to prove corresponding results for the homogeneity with respect to a wider class of analytic diffeomorphisms of $\mathbb R^2$.  

Of course, we should expect the homogeneity with respect to the full group of real-analytic diffeomorphisms (rather than certain specific subgroups) to be sufficient to conclude that the curve is real-analytic.  Our method at the present stage needs some additional rigidity of the diffeomorphism groups in question; on the other hand, one off-shot is the consideration of some interesting dynamical questions.

Indeed, the verification of property (ICP) leads to the study of the analytic isotropy group of non-analytic curves, a problem that we find interesting in its own right and seems not to be completely understood. We are able to check that (ICP) holds for certain groups whose dynamics is relatively simple: the first instance is given by Lemma \ref{leaufatou} stated above.

Let us now consider the group ${\rm Shr}^\omega(\mathbb R^2,0)$ of \emph{analytic shears}, that is the germs $\psi$ of diffeomorphisms of the form $\psi(x,y)= (\phi(x), y + h(x))$ for certain $\phi, h$, where $\phi\in {\rm Diff}^\omega(\mathbb R,0)$ and $h$ is a germ of real-analytic function at $0$. Also for the case of shears we obtain 
\begin{lemma}\label{invshr}
The group ${\rm Shr}^\omega(\mathbb R^2,0)$ satisfies (ICP). More precisely we have the following: let $\phi\in {\rm Shr}^\omega(\mathbb R^2,0)$, $\phi$ not involutive, and let $\gamma\subset \mathbb R^2$ be a $\phi$-invariant germ of curve around $0$, which is of the form $\gamma = \{y = f(x)\}$ for some continuous function $f$. Then $\gamma\setminus \{0\}$ is real-analytic; moreover, $\gamma$ is uniquely determined.
\end{lemma}
from which we get the following result:

\begin{theorem}\label{shearhomo}
Let $\gamma\subset \mathbb R^2$ be a smooth  curve. Then $\gamma$ is ${\rm Shr}^{\omega}(\mathbb R^2,0)$-homogeneous if and only if it is either of the form $\gamma = \{x = c\}$, $c\in \mathbb R$, or $\gamma=\{y=f(x)\}$ for a real-analytic function $f$.
\end{theorem}

Although Theorem \ref{shearhomo} is about a group that might not be considered very \lq\lq natural\rq\rq\, it can be applied to obtain the following corollary (which does not follow from the holomorphic version):

\begin{cor}\label{diffquot}
Let $I\subset \mathbb R$, $0\in I$, be an interval and let $f\in C^\infty(I)$. Suppose that, for all fixed $t\in I$, the function $g_t(x) = f(x+t) - f(x)$ is real-analytic on a neighborhood $I_t$ of $0$ (depending on $t$). Then $f\in C^\omega(I)$.
\end{cor}

We can also prove the analogous results for affine transformations:
\begin{lemma}\label{linicp}
The group $GL(2,\mathbb R)$ satisfies (ICP). More precisely, let $A\in GL(2,\mathbb R)$, $A$ not involutive, and let $\gamma$ be an invariant curve of class $C^\infty$ for $A$ passing through $0$. Then $\gamma$ is real-analytic (in fact real-algebraic).
\end{lemma}
The proof in this case can be achieved without referring to the theory of dynamical systems. As before we get as a corollary
\begin{theorem}\label{linhomo}
Let $\gamma\subset \mathbb R^2$ be a curve of class $C^\infty$, locally homogeneous under affine transformations. Then $\gamma$ is real-analytic.
\end{theorem}
More precisely, the method shows that a curve $\gamma$ which is locally homogeneous under affine transformations is the integral curve of an affine vector field. It does not follow, however, that $\gamma$ is a conic section, or even algebraic: for instance the curve $\gamma=\{y=e^x\}$ is homogeneous under affine transformations (one can see this by considering the affine diffeomorphisms $\psi_t(x,y)=(x+t, e^ty)$, as indeed $\gamma$ is an integral curve for the vector field $\frac{\partial}{\partial x} + y\frac{\partial}{\partial y}$).

Perhaps a bit surprisingly, these results do not extend to the full group ${\rm Diff}^\omega(\mathbb R^2,0)$: indeed, there exist nowhere analytic planar curves of class $C^\infty$, with divergent Taylor expansion, whose analytic isotropy group is infinite (in \cite{De}, examples of such curves are constructed and the conditions for this to happen are studied). Thus, the procedure we adopt here cannot be directly generalized to the full group of real-analytic diffeomorphisms. It seems likely, however, that a better understanding of the isotropies of nowhere analytic curves 
would be useful for this approach.

\

\subsection*{Organization of the paper}

\

\

The proofs of the main results are contained in section \ref{homo}. In section \ref{defnot} we introduce the notation and certain definitions that we will employ in the rest of the paper; some of these might not be completely standard, for instance our use of the term \lq\lq analytic shear\rq\rq. In section \ref{Cartan} we review some results from the paper by H. Cartan \cite{Ca} mentioned above, which will be needed for the construction of a non-trivial vector field. Finally, in section \ref{lcs} we discuss the case of analytically homogeneous, locally closed sets. As a consequence of Theorem \ref{Wilkin} from \cite{RSS}, \cite{Sk} and \cite{Wi} we obtain the following

\begin{theorem}\label{disopcur}
Let $\mathcal G\leq {\rm Diff}^\omega(\mathbb R^2,0)$ be a subgroup satisfying the assumptions of Proposition \ref{Ghom}. Then every $\mathcal G$-homogeneous, locally closed subset $K\subset \mathbb R^2$ is either discrete, an open subset, or a real-analytic curve.
\end{theorem}

\section{Definitions and notation} \label{defnot}
\subsection{Sets of analytic diffeomorphisms}
Let $U\subset \mathbb R^n$ be an open subset, and let $\psi: U \to \mathbb R^n$ be a map of class at least $C^1$. For any $p\in U$, we denote by $J\psi(p)$ the Jacobian determinant of $\psi$ at the point $p$. We shall denote by ${\rm Diff}(U)$ the set of all maps $\psi:U \to \mathbb R^n$ of class $C^1$ such that $J\psi(p)\neq 0$ for all $p\in U$. We will equip ${\rm Diff}(U)$ with the topology of uniform convergence on compact subsets of $U$.
For any $k\in \mathbb N \cup \{\infty, \omega\}$, we employ the notation ${\rm Diff}^k(U)$ for the set of maps $\psi\in {\rm Diff}(U)$ which are of class $C^k$. Moreover, ${\rm Diff}^k_*(U)$ will stand for the set of all the mappings $\psi\in {\rm Diff}^k(U)$ which admit an inverse $\psi^{-1}\in {\rm Diff}^k(\psi(U))$.

Given an open subset $U\subset \mathbb C^n$, we define the sets ${\rm Hol}(U)$, ${\rm Hol}_*(U)$ in an analogous way in the category of holomorphic maps $\psi: U\to \mathbb C^n$.

We will always denote the identity map by $Id$, whether in the context of analytic diffeomorphisms or as the corresponding element in a group of germs. Moreover, given a map $\psi$ and $k\in \mathbb N$, we denote by $\psi^{\circ k}$ the composition of $\psi$ with itself performed $k$ times (we are borrowing this notation from \cite{Br}). The domain of the iteration will in general depend on $k$, and we include the possibility of it being the empty set. If $\psi$ is invertible we also extend this notation to $k\in \mathbb Z$, with the convention $\psi^{\circ 0} = Id$.

\subsection{Groups of germs of analytic maps}\label{germanmaps}
Fix, now, $p\in \mathbb R^n$. We denote by ${\rm Diff}^\omega(\mathbb R^n,p)$ the group of germs at $p$ of real-analytic maps $\psi:U\to \mathbb R^n$, where $U$ is an open neighborhood of $p$ in $\mathbb R^n$, such that $\psi(p)=p$ and $J\psi(p)\neq 0$. If $p\in \mathbb C^n$, we define ${\rm Hol}(\mathbb C^n,p)$ in a similar way. 

Endow ${\rm Diff}^\omega(\mathbb R^n,p)$ with its topology as a direct limit, that is, given $\{h_j\}_{j\in \mathbb N}\subset {\rm Diff}^\omega(\mathbb R^n,p)$ and $h\in {\rm Diff}^\omega(\mathbb R^n,p)$ we have $h_j\to h$ if and only if there is a neighborhood $U$ of $p\in \mathbb R^n$ and analytic diffeomorphisms $\tilde h_j:U\to \tilde h_j(U)\subset \mathbb R^n$, $\tilde h:U\to \tilde h(U)\subset \mathbb R^n$ such that the germ induced by $\tilde h_j$ (resp. $\tilde h$) at $p$ is  $h_j$ (resp. $h$) and $h_j\to h$ uniformly on $U$. Let $\mathcal G\leq {\rm Diff}^\omega(\mathbb R^n,p)$ be a subgroup. Then $\mathcal G$ is a closed subgroup if the following property holds: 

\emph{Let $g\in {\rm Diff}^\omega(\mathbb R^n,p)$ be the germ at $p$ of a real-analytic map $g:U\to\mathbb R^n$, and suppose that there exist a domain $U'\subset U$ and a sequence $\{g_j\}_{j\in \mathbb N}$ of maps $U'\to \mathbb R^n$ such that
\begin{itemize}
\item the germ induced by $g_j$ at $p$ belongs to $\mathcal G$ for each $j\in \mathbb N$;
\item $g_j\to g$ as $j\to \infty$, uniformly over $U'$.
\end{itemize}
Then $g\in \mathcal G$.}

With this definition, if $n=2m$ and $p\in \mathbb R^n \cong \mathbb C^m$ the group $\mathcal G= {\rm Hol}(\mathbb C^m,p)$ can be naturally identified with a closed subgroup of ${\rm Diff}^\omega(\mathbb R^n,p)$.  Other natural examples of closed subgroups of ${\rm Diff}^\omega(\mathbb R^n,0)$ are provided by $GL(n,\mathbb R)$ and any closed (in the topology of Lie groups) $G\leq GL(n,\mathbb R)$.

We will also be interested in another, somewhat more artificial subgroup of analytic maps, depending on the choice of a special direction in $\mathbb R^n$. Fix coordinates in $\mathbb R^n$ of the form $(x,y)$, where $x=(x_1,\ldots, x_{n-1})\in \mathbb R^{n-1}$ and $y\in \mathbb R$. Given an open subset $U\subset \mathbb R^n$ of the form $U=U_1\times \mathbb R$, $U_1\subset \mathbb R^{n-1}_x$, we call a map $\psi\in {\rm Diff}^\omega(U)$ a \emph{(real-analytic) shear} if for all $(x,y)\in U$ we have
\[\psi(x,y)= (\phi(x), y + h(x))\]
where $\phi\in {\rm Diff}^\omega(U_1)$ and $h\in C^\omega(U_1)$. We denote the class of analytic shears on $U$ as ${\rm Shr}^\omega(U)$, and for $p\in \mathbb R^n$ we define the set of germs ${\rm Shr}^\omega(\mathbb R^n,p)$ accordingly. It is straightforward to check that the set ${\rm Shr}^\omega(\mathbb R^n,p)$ constitutes a subgroup of ${\rm Diff}^\omega(\mathbb R^n,p)$, and that moreover it is a closed subgroup in the sense specified above. Algebraically, ${\rm Shr}^\omega(\mathbb R^n,0)$ is isomorphic to the semidirect product ${\rm Diff}^\omega(\mathbb R^{n-1},0) \ltimes_{\varphi} C^\omega(\mathbb R^{n-1},0)$, where $C^\omega(\mathbb R^{n-1},0)$ is regarded as an additive group and the homomorphism $\varphi:{\rm Diff}^\omega(\mathbb R^{n-1},0) \to {\rm Aut}(C^\omega(\mathbb R^{n-1},0))$ is given by $\varphi(\psi) f = f \circ \psi$ for $\psi \in{\rm Diff}^\omega(\mathbb R^{n-1},0)$, $f \in C^\omega(\mathbb R^{n-1},0)$.

For any $q\in \mathbb R^n$, denote by $\tau_q:\mathbb R^n\to\mathbb R^n$ the translation $\tau_q(x) = x + q$. We will say that a subgroup $\mathcal G\leq {\rm Diff}^\omega(\mathbb R^n,0)$ is \emph{uniform} if the following condition is satisfied: for any domain $U\subset \mathbb R^n$, $0\in U$, and any $\psi \in {\rm Diff}^\omega(U)$ such that the germ of $\psi$ at $0$ belongs to $\mathcal G$, for all $q\in U$ the germ at $0$ of the map $ \tau_{-\psi(q)} \circ \psi \circ \tau_{q}$ also belongs to $\mathcal G$. In other words, we require the representatives of the germs of $\mathcal G$ to still induce a germ of $\mathcal G$ at any point of their domain (after suitable translations). It is immediate to check that ${\rm Hol}(\mathbb C^m,0)$, $GL(n,\mathbb R)$ and ${\rm Shr}^\omega(\mathbb R^n,p)$ all conform to this definition.

Finally, we will need to single out the subset $\mathcal R\subset {\rm Diff}^\omega(\mathbb R^n,0)$ of the germs of order two (or \emph{involutions}), i.e\ $\psi\in \mathcal R \Leftrightarrow \psi^{\circ 2} = Id$ (note that $\mathcal R$ is not  a subgroup of ${\rm Diff}^\omega(\mathbb R^n,0)$). It is well-known that each element of $\mathcal R$ is conjugated to its linear part: indeed, let $\psi\in \mathcal R$ and put $A = d\psi(0)$. We define the germ $\varphi$ of an analytic mapping around $0$ by putting $\varphi(x) = x + A\psi(x)$. Since $d\varphi(0) = Id + A^2 = 2Id$, $\varphi$ is actually a local diffeomorphism. Furthermore, $\varphi \circ \psi = \psi + A \psi^{\circ 2} = \psi + A$ and $A\circ \varphi = A(Id + A\psi) = A + A^2\psi = A + \psi$, so that $A = \varphi\circ \psi \circ \varphi^{-1}$. In particular, in dimension $n=2$ an element $\psi\in \mathcal R$ either coincides with the identity or is conjugated to $(x,y) \to (-x,y)$ or $(x,y)\to (-x,-y)$. In each case it is clear that $\psi$ admits nowhere analytic invariant curves of class $C^\infty$: that is the reason for which the elements of $\mathcal R$ need to be excluded in the lemmas of section \ref{aninvcur}.

\subsection{Local equivalence problem}\label{locequivpro}
Let $p=0\in \mathbb R^n$ and let $\mathcal G$ be a subgroup of ${\rm Diff}^\omega(\mathbb R^n,0)$. Let $(M_1,0)$, $(M_2,0)$ be germs at $0$ of smooth submanifolds of $\mathbb R^n$. We say that such germs are \emph{(locally) equivalent under $\mathcal G$}, and we write
\[(M_1,0) \stackrel{\mathcal G}{\sim}(M_2,0)\]
if there exist neighborhoods $U_1,U_2$ of $0$ in $\mathbb R^n$, representatives $M_1\subset U_1$, $M_2\subset U_2$ of the two germs of manifolds, and a map $\psi\in {\rm Diff}^\omega(U_1)$ whose germ induced at $0$ belongs to $\mathcal G$ such that $\psi(M_1)=M_2$. Since $\mathcal G$ is taken to be a subgroup, $\stackrel{\mathcal G}{\sim}$ is an equivalence relation on the set of germs $(M,0)$. If there is no ambiguity on the group which is referred to, we will just denote the equivalence relation by $\sim$ (typically in the case when $\mathcal G = {\rm Diff}^\omega(\mathbb R^n,0)$).

It is worth remarking that, considering the standard embedding $\mathbb R^n\subset \mathbb C^n$, two manifold germs $(M_1,0)$ and $(M_2,0)$ contained in $\mathbb R^n$ are equivalent under ${\rm Diff}^\omega(\mathbb R^n,0)$ if and only if (seen as germs of manifolds in $\mathbb C^n$) they are equivalent under ${\rm Hol}(\mathbb C^n,0)$. This can be deduced from the following more general lemma: consider, in $\mathbb R^n = \mathbb R^m \times \mathbb R^{n-m}$, coordinates $x=(x',x'')$, where $x'=(x_1,\ldots x_m)$, $x''=(x_{m+1},\ldots, x_n)$. Moreover, denote by $\pi', \pi''$ the projections defined by $\pi'(x) = x', \pi''(x)=x''$.

\begin{lemma}\label{subspace}
Let $(M_1,0)$, $(M_2,0)$ be germs of submanifolds of $\mathbb R^m$ around $0$. Then $M_1$ and $M_2$ are equivalent under ${\rm Diff}^\omega(\mathbb R^n,0)$ if and only if they are equivalent under ${\rm Diff}^\omega(\mathbb R^m,0)$.
\end{lemma}
\begin{proof}
If $\phi\in {\rm Diff}^\omega(\mathbb R^m,0)$ is such that $\phi(M_1) = M_2$, then the map $\psi(x) = (\phi(x'),x'')$ belongs to ${\rm Diff}^\omega(\mathbb R^n,0)$ and is a local equivalence between $M_1$ and $M_2$.

Vice versa, let $\psi$ be an element of ${\rm Diff}^\omega(\mathbb R^n,0)$ that sends $M_1$ to $M_2$, and let $\psi_0 = \psi_{|{\mathbb R^m}}: \mathbb R^m \to \mathbb R^n$. We write the differential of $\psi_0$ at $0$ in the following way:
\[
d\psi_0(0)=\begin{pmatrix}
			(d\psi_0)'(0) \\
			(d\psi_0)''(0)
\end{pmatrix}, 
\]
where $(d\psi_0)'(0)$ and $(d\psi_0)''(0)$ are, respectively, the blocks of $d\psi_0(0)$ of dimensions $m\times m$ and $(n-m)\times m$. We note that, since $\psi$ is a diffeomorphism, $d\psi_0(0)$ has rank $m$.

For any $m\times(n-m)$-matrix $A$ with real coefficients, we consider the linear map $\sigma_A(x) = (x' + Ax'', x'')$; notice that $\sigma_A\in SL(n,\mathbb R)$, its restriction to $\mathbb R^m$ is the identity, and that $\sigma_{A_2}\sigma_{A_1} = \sigma_{A_1 + A_2}$ for any two $m\times(n-m)$-matrices $A_1,A_2$. The differential of $\sigma_A\circ \psi_0$ at the origin can be expressed as follows:
\[
d(\sigma_A\circ \psi_0)(0) = \begin{pmatrix}
			(d\psi_0)'(0) + A\cdot (d\psi_0)''(0)  \\
			(d\psi_0)''(0) 
\end{pmatrix}.
\]
Since $d\psi_0(0)$ has rank $m$  we can choose $A$ in such a way that $(d\psi_0)'(0) + A\cdot (d\psi_0)''(0)$ is non-singular. Indeed, suppose that $(d\psi_0)'(0)$ has rank $\ell<m$; up to permutation of the coordinates, we can assume that the first $\ell$ rows of $(d\psi_0)'(0)$, together with the first $m-\ell$ rows of $(d\psi_0)''(0)$, are $m$ independent vectors. The matrix $A = (a_{ij})$ ($1\leq i \leq m$, $1\leq jÊ\leq n-m$) can be defined by setting $a_{ij} = 1$ whenever $j -i = \ell$ and $a_{ij} = 0$ otherwise.

With such a choice of $A$, we define $\phi:\mathbb R^m\to \mathbb R^m$ by $\phi = \pi' \circ \sigma_A \circ \psi_0$. Since $d\phi(0) = (d\psi_0)'(0) + A\cdot (d\psi_0)''(0)$ is non-singular, $\phi$ belongs to ${\rm Diff}^\omega(\mathbb R^m,0)$; moreover, since $\pi'_{|{\mathbb R^m}} = (\sigma_A)_{|{\mathbb R^m}} = Id$, we have (locally) $\phi(M_1) = M_2$, thus $M_1$ and $M_2$ are locally equivalent under ${\rm Diff}^\omega(\mathbb R^m,0)$.
\end{proof}
If now two germs $(M_1,0),(M_2,0)\subset \mathbb R^n\subset \mathbb C^n$ are ${\rm Hol}(\mathbb C^n,0)$-equivalent, they are in particular ${\rm Diff}^\omega(\mathbb R^{2n},0)$-equivalent, hence by Lemma \ref{subspace} they are equivalent under ${\rm Diff}^\omega(\mathbb R^n,0)$. On the other hand, a germ of map $\psi\in {\rm Diff}^\omega(\mathbb R^n,0)$ can always be locally extended to a $\Psi\in {\rm Hol}(\mathbb C^n, 0)$ such that $\Psi|_{\mathbb R^n}=\psi$. So, in this sense, the local equivalence problem under ${\rm Diff}^\omega(\mathbb R^n,0)$ is no more general than the problem in the holomorphic category.

\subsection{Equivalence loci}
Let $M\subset U \subset \mathbb R^n$ be a smooth submanifold of codimension $d$, and let $\rho:U\to \mathbb R^d$, $\rho(x) = (\rho_1(x),\ldots, \rho_d(x))$ be a (vector-valued) defining equation for $M$ of class $C^\infty$. For any $p\in M$, we define $\rho_p(x)=\rho(x+p)$ and $M_p = \{\rho_p = 0\}$. The manifold $M_p$ is of course just the image of $M$ through the translation $\tau_{-p}: x\to x-p$ taking $p$ to $0$, and the germ $(M_p,0)$ of $M_p$ at $0$ is the translate of the germ $(M,p)$ of $M$ at $p$. Therefore, the definition of $(M_p,0)$ does not depend on the choice of the defining equation $\rho$.

Fix $p\in M$ and a subgroup $\mathcal G\subset {\rm Diff}^\omega(\mathbb R^n,0)$. We are interested in the following set:
\[E_p^{\mathcal G} = \{q\in M : (M_q, 0)\stackrel{\mathcal G}{\sim} (M_p,0)  \}. \]
We call $E_p^{\mathcal G}$ the \emph{$\mathcal G$-equivalence locus} of $p$ in $M$. In other words, $q\in E_p^{\mathcal G}$ if and only if there exist a neighborhood $U$ of $p$ in $\mathbb R^n$ and a map $\psi\in {\rm Diff}^\omega(U)$ such that $\psi(M)\subset M$, $\psi(p) = q$ and the germ induced at $0$ by the map $\tau_{-q} \circ \psi \circ \tau_p$ belongs to $\mathcal G$.

If for any (and hence for all) $p\in M$ we have $E_p^{\mathcal G} = M$, we say that $M$ is \emph{$\mathcal G$-homogeneous}. When $\mathcal G = {\rm Diff}^\omega(\mathbb R^n,0)$ (or when there is no confusion on the group $\mathcal G$ under consideration) we simply write $E_p$ and we call it the \emph{equivalence locus} of $p$; accordingly, we say that a manifold $M$ is \emph{homogeneous} if it is ${\rm Diff}^\omega(\mathbb R^n,0)$-homogeneous. If $\mathcal G = {\rm Diff}^k(\mathbb R^n,0)$, $k\in \mathbb N$, we also refer to the $\mathcal G$-homogeneity as \emph{$C^k$-homogeneity}.

If $\mathcal G = \{Id\}$, the locus $E_p^{\{Id\}}$ just corresponds to the set of the translations $\tau_q$ ($q\in M$) which bring a neighborhood $U_q$ of $p$ in $M$ into a neighborhood of $q$ in $M$. The structure of the equivalence locus can be, nevertheless, non-trivial even in this case (examples of \lq\lq complicated behavior\rq\rq\ of the equivalence locus are studied in \cite{De2}); this is due to the fact that the neighborhood $U_q$ is allowed to shrink as $q$ varies in $M$.

We point out that a generalization of this notion of equivalence locus, referred to the biholomorphic equivalence up to the $k$-th order of  real-analytic germs for every $k\in \mathbb N$, was introduced in \cite{Za} under the name of \emph{weak equivalence orbit}. For a real-analytic manifold, homogeneity according to this weaker notion is sufficient to recover the strongest possible homogeneity properties like the existence of a local transitive action of a Lie group by CR automorphisms (see \cite[Theorem 1.4]{Za}). This is no longer true in the $C^\infty$ case: any smooth curve $\gamma\subset \mathbb R^2\subset \mathbb C^2$ is homogeneous in the weak sense but, of course, needs not be ${\rm Diff}^\omega(\mathbb R^2,0)$-homogeneous in general.

\section{Local groups of transformations}\label{Cartan}
In this section we are going to briefly outline some results that we will need later for the proof of Proposition \ref{Ghom}, coming from the work of H.Cartan \cite{Ca} on (pseudo)groups of holomorphic mappings. One reason for doing so (as opposed to just citing the paper) is that the language employed in \cite{Ca} does not completely coincide with the modern terminology, in particular concerning some topological or complex-analytic concepts. The main aim of the paper in question is to prove that the set of holomorphic automorphisms of a bounded domain $D\subset \mathbb C^n$ is a finite dimensional Lie group; we will, however, be interested in some of the intermediate results, which involve mappings which are not necessarily automorphisms of $D$ (or even valued in $D$) and thus fit into our framework.

Let $\mathcal E$ be a topological manifold admitting a metric $\delta$ (usually $\mathcal E=\mathbb R^n, \mathbb C^m$) and let $D\subset \mathcal E$ be a subdomain. Consider a set $G$ of continuous mappings $D\to \mathcal E$ satisfying the following condition:
\begin{itemize}
\item[(a)] Fix two relatively compact (\lq\lq completely interior\rq\rq\ in the terminology of \cite{Ca}) subdomains $\Delta,\Delta' \Subset D$. Then for every $\epsilon>0$ there exists $\epsilon'>0$ such that, for all $\psi_1,\psi_2\in G$, $\sup\{\delta(\psi_1(x),\psi_2(x)): x\in \Delta\}<\epsilon$ whenever $\sup\{\delta(\psi_1(x), \psi_2(x)): x\in \Delta'\}<\epsilon'$.
\end{itemize}
In this case, we consider in $G$ the topology of uniform convergence in $\overline \Delta_0$ for an arbitrary, fixed subdomain $\Delta_0\Subset D$, and for any $\psi_1,\psi_2\in G$, $\Delta\Subset D$ we define $\delta_{\Delta}(\psi_1,\psi_2) := \sup\{\delta(\psi_1(x),\psi_2(x)) : x\in \Delta\}$, $\delta_\Delta(\psi_1) := \delta_\Delta(\psi_1,Id)$. The set $G$ is then called a \emph{transformation group} on $D$, seen \emph{from the local point of view}, if, furthermore, $Id\in G$ and the following conditions are fulfilled:
\begin{itemize}
\item[(b)] There is a neighborhood $G'$ of $Id$ in $G$ and a law that associates to any pair $(\psi_2,\psi_1)\in G'^2$ an element $\phi = \phi(\psi_2,\psi_1)\in G$, satisfying the following properties. For every domain $\Delta\Subset D$ there exists $\eta(\Delta)>0$ such that, for all $\psi_1\in G$ for which $\delta_\Delta(\psi_1)<\eta(\Delta)$ we have
					\begin{itemize}
						\item[(1)] $\psi_1\in G'$;
						\item[(2)] $\psi_1(\Delta)\subset D$;
						\item[(3)] for all $x\in \Delta$ and all $\psi_2\in G'$, $\phi(x) = \psi_2(\psi_1(x))$.
					\end {itemize}
We write $\phi(\psi_2,\psi_1) = \psi_2\circ \psi_1$.
\item[(c)] There is a neighborhood $G''\subset G'$ of $Id$ in $G$ such that, for all $\psi_1\in G''$, there exists $\psi_2\in G'$ such that $\psi_2\circ \psi_1 = Id$. We write $\psi_2 = \psi_1^{-1}$. 
\end{itemize}

Suppose that $D\subset \mathbb C^n$ and let $G$ be a subset of ${\rm Hol}(D)$ which satisfies the conditions (b) and (c) above and, moreover, is uniformly bounded over $D$. Then by Montel's theorem $G$ also satisfies (a) and it is thus a local group of transformations according to the previous definition (a local group of \emph{pseudo-conformal} transformations in the language of Cartan).

Let then $G$ be a (local) group of analytic transformations on a domain $D$. We say that $G$ satisfies the \emph{property} [\emph{P}] if 
\begin{itemize}
\item[{[P]}] for every sequence $\{S_i\}_{i\in \mathbb N}\subset G$ such that $S_i\to Id$ as $i\to \infty$, there exists a subsequence $\{S_{k_i}\}$ and positive integers $m_i$ such that $m_i(S_{k_i} - Id)$ converges uniformly on compact subsets of $D$ to a non-identically-vanishing, analytic map $\psi$.
\end{itemize}
The importance of the previous condition lies in the fact that it allows to construct one-parameter groups of analytic transformations contained in $G$. In fact, the map $\psi$ turns out to be to be an infinitesimal generator, in the sense that its components can be interpreted as the components of a vector field whose flow, computed for small real parameters, is contained in (the closure of) $G$:
\begin{theorem}[{\cite[Th\'eor\`eme 3]{Ca}}]\label{Th3}
Let $G$ be a locally closed group of analytic transformations. If there exist sequences $\{S_i\}_{i\in \mathbb N}\subset G$, $S_i\to Id$, and  $\{m_i\}_{i\in \mathbb N}\subset \mathbb N$ such that $m_i(S_i - Id)\to \psi$ uniformly (where $\psi$ is a  non-identically-vanishing analytic map) then the group $G$ admits the infinitesimal transformation $\varphi(\xi,t)$ defined by
\begin{equation} \label{infitra}
 \frac{\partial \varphi}{\partial t}(\xi,t) = \psi(\varphi(\xi,t)),
 \end{equation}
 i.e.\ $\varphi(\cdot,t)\in G$ for $t$ in a neighborhood of $0$ in $\mathbb R$.
\end{theorem}

\begin{remark}
In the previous statement, the assumption that $G$ is locally closed regards $G$ as a subset of ${\rm Diff}^\omega(D)$, with the topology of uniform convergence on compact subsets of $D$. In Cartan's paper, the definition of \lq\lq locally closed\rq\rq\ actually corresponds to the current notion of local compactness. An inspection of the proof of Theorem \ref{Th3}, however, shows that the relevant property (only for the purpose of proving that result) is local closeness. In any case, our aim is to apply this result in the context of local groups of \emph{holomorphic} transformations, in which, by Montel's theorem, locally closed implies locally compact.
\end{remark}

\begin{remark}
The precise meaning of (\ref{infitra}) is the following. $G$ is a group of transformations on a domain $D\subset \mathbb R^n$ or $\mathbb C^n$, and the variable $\xi=(\xi_1,\ldots, \xi_n)$ represents a $n$-uple of real or complex coordinates on $D$. We think of the map $\psi = (\psi_1,\ldots, \psi_n)$, whose components are real-analytic or holomorphic in $\xi$, as the vector-field $\Psi = \sum_j \psi_jÊ\partial/\partial \xi_j$. Then (\ref{infitra}) represents the system $\partial \varphi_j(\xi,t)/\partial t = \psi_j(\varphi(\xi,t))$ ($j=1,\ldots, n$) with initial condition $\varphi(\xi,0)=\xi$, where $\varphi = (\varphi_1,\ldots,\varphi_n)$ and $t$ is a \emph{real} variable. For every $K\Subset D$, there exists $\delta>0$ such that the (unique) solution $\varphi$ of this system is defined for $\xi\in K$ and $-\delta<t<\delta$; $\varphi(\xi,t)$ is real-analytic or holomorphic with respect to $\xi$ and real-analytic in $t$. 

In the holomorphic case (where $D\subset \mathbb C^n$, $\xi_j = x_j + i y_j$), breaking down (\ref{infitra}) in its real and imaginary part one can check that $\varphi$ is the flow of the vector field ${\rm Re} \Psi = \sum_j {\rm Re} \psi_j \partial/\partial x_j + {\rm Im} \psi_j \partial/\partial y_j$ (note that $\Psi$ itself is a section of $\mathbb C \otimes T(D)$ rather than $T(D)$). In this case, by classical results we can also integrate the complex version of the system, i.e. we can locally find a holomorphic solution $\Phi$ to the system $\partial \Phi(\xi,\tau)/\partial \tau = \psi(\Phi(\xi,\tau))$, where $\tau = t + is$ is now a complex variable. This produces the same outcome, that is $\Phi(\xi,t) = \varphi(\xi,t)$ for $t\in \mathbb R$; indeed, since by the Cauchy-Riemann conditions $\partial \Phi/ \partial t = -i \partial \Phi/\partial s$, we have $\partial \Phi/ \partial \tau = \frac{1}{2} (\partial \Phi/ \partial t - i\partial \Phi/ \partial s) = \partial \Phi/ \partial t$ and thus $\Phi(\xi,t)$ is, too, a solution to (\ref{infitra}). 
\end{remark}
The second ingredient that we recall from \cite{Ca} is the following result:

\begin{theorem}[{\cite[Th\'eor\`eme 10]{Ca}}]\label{Th10}
All the groups of holomorphic transformations seen from the local point of view satisfy property [P].
\end{theorem}
Putting together Theorems \ref{Th3} and \ref{Th10} we deduce immediately
\begin{cor}\label{puttog}
Let $G$ be a locally closed group of holomorphic transformations, seen from the local point of view, such that $Id$ is not isolated. Then $G$ admits a non-trivial infinitesimal transformation.
\end{cor}
Corollary \ref{puttog} will be used in Section \ref{proofpro} for the proof of Proposition \ref{Ghom}.

\section{Homogeneous curves}\label{homo}

\subsection{Uniformity of the domain}

Let $M\subset\mathbb R^n$ be a ($\mathcal G$-)homogeneous submanifold; for every $p,q\in M$ there is a real-analytic diffeomorphism $\psi$, defined on a neighborhood $U$ of $p$ in $\mathbb R^n$, such that $\psi(p)=q$ and $\psi(M\cap U) = M\cap \psi(U)$. However, the domain $U$ on which $\psi$ is defined depends on $p$ and $q$. Our first step is to show that suitably selected families of transformations are defined on a fixed domain. We will prove this fact in Lemma \ref{Baire} using  Baire's theorem; however we first need a (probably well-known) lemma about the existence of the inverse of a map which is the uniform limit of holomorphic diffeomorphisms.

\begin{lemma}\label{limitinv}
Let $U\subset \mathbb C^n$ be a relatively compact domain, and let $\{\psi_n\}_{n\in \mathbb N}$ be a sequence of holomorphic mappings $\psi_n: U\to \mathbb C^n$ such that 
\begin{itemize}
\item each $\psi_n$ admits a holomorphic inverse $\psi_n^{-1}:\psi_n(U)\to U$;
\item there exists $\epsilon>0$ such that $|J\psi_n(z)|\geq \epsilon$ for all $n\in \mathbb N$, $z\in U$;
\item $\{\psi_n\}$ converges, uniformly over $U$, to a mapping $\psi:U\to \mathbb C^n$.
\end{itemize}
Then $\psi$ admits a holomorphic inverse $\psi^{-1}:\psi(U)\to U$.
\end{lemma}
\begin{proof}
Note that by the Cauchy estimates follows that $\psi_n\to \psi$ uniformly with all the derivatives on compact subsets of $U$. Thus, by continuity we have that $|J\psi(z)|\geq \epsilon$ for all $z\in U$, hence $\psi$ is a local diffeomorphism. In particular $\psi(U)$ is an open domain of $\mathbb C^n$. Let $K\subset \psi(U)$ be any compact subset; we claim that there exist a domain $V\supset K$, $V\Subset \psi(U)$, and $n_0\in \mathbb N$ such that $\psi_n(U)\supset V$ for all $n\geq n_0$. 

Let $q\in K$; by compactness, to verify the claim it is sufficient to show that for a small enough $r>0$ there exists $n_0\in \mathbb N$ such that $B_r(q)\subset \psi_n(U)$ for all $n\geq n_0$. Choose $p\in U$ such that $\psi(p)=q$; since $\psi$ is a local diffeomorphisms, there exists $U_1\subset U$, $p\in U_1$ such that $\psi|_{U_1}$ has a local inverse $\psi^{-1}: \psi(U_1)\to U_1$. If now we choose $r_1>0$ such that $B_{r_1}(p) \subset U_1$, for $n$ large enough we have $\psi_n(B_{r_1}(p))\subset \psi(U_1)$, hence we can define a mapping $\phi_n: B_{r_1}(p)\to \mathbb C^n$ as $\phi_n = \psi^{-1}\circ \psi_n$. By assumption we have $\phi_n\to Id$ uniformly as $n\to \infty$, thus it follows (for example by applying the several-variables version of Rouch\'e's theorem, see \cite{Ll}, or Hurwitz's theorem) that $\phi_n(B_{r_1}(p))$ contains a fixed neighborhood $U_2$ of $p$ for all $n\geq n_0$ for a certain $n_0\in \mathbb N$. We deduce that $\psi_n(B_{r_1}(p))\supset \psi(U_2)$ for $n\geq n_0$, which implies the claim.

Choose an exhaustion $\{K_j\}_{j\in \mathbb N}$ of $\psi(U)$ by compact subsets, and let $\{V_j\}_{j\in \mathbb N}$, $K_j\subset V_j$, be the associated domains as found above; moreover, define $U_j = \psi^{-1}(V_j)\subset U$. Note that, since each mapping $\psi_k^{-1}$ is valued in $U$, for any $k_0\in \mathbb N$ the family $\{\psi_k^{-1}\}_{k\geq k_0}$ is uniformly bounded on any common domain of definition $D$ (i.e.\ any domain $D\subset \bigcap_{k\geq k_0} \psi_k(U)$).  
 Then, using Montel's theorem and a suitable diagonal argument, we can select a subsequence of $\{\psi_{n}^{-1}\}$, which we still denote in the same way, such that $\psi_n^{-1}$ is eventually defined on any $V_j$ and $\psi_n^{-1}\to \phi$ as $n\to\infty$ uniformly on compact subsets of $\psi(U)$, where $\phi: \psi(U)\to U$ is a holomorphic mapping.
 
Let $p\in U$, and pick $j\in \mathbb N$ such that $p\in U_j$. Since the mappings $\psi_n$ converge to $\psi$, we also have $\psi_n(p)\in V_j$ for all large enough $n$. Then from the fact that $\psi_n^{-1}\to \phi$ uniformly on $V_j$ follows that $\psi_n^{-1} \circ \psi_n(p) \to \phi\circ \psi(p)$ as $n\to \infty$, implying $\phi\circ \psi(p) = p$. Hence the mapping $\psi$ is injective and $\phi=\psi^{-1}$ is its holomorphic inverse.
\end{proof}

\begin{lemma}\label{Baire}
Let $M\subset \mathbb R^n$ be a homogeneous submanifold. Then there exist $p_0\in M$, a neighborhood $U$ of $p_0$ in $\mathbb C^n$ and a family $\{\phi_q\}\subset {\rm Hol}_*(U)$ ($q\in M\cap U$) of holomorphic mappings with the following properties:
\begin{itemize}
\item the restriction of $\phi_q$ to $\mathbb R^n$ is an element of ${\rm Diff}_*^\omega(U\cap \mathbb R^n)$;
\item $\phi_q(M\cap U)\subset M$ for all $q\in M\cap U$;
\item $\phi_q(p_0)=q$ for all $q\in M\cap U$, and $\phi_{p_0}\equiv Id$;
\item the maps $\phi_q$ are uniformly bounded, and $|J\phi_q|$ bounded below by a positive constant, as $q\in M\cap U$.
\end{itemize} 
If, moreover, $M$ is $\mathcal G$-homogeneous for a closed subgroup $\mathcal G \leq {\rm Diff}^\omega(\mathbb R^n,0)$, the family $\{\phi_q\}$ can be chosen in such a way that, for all $q\in M\cap U$, the germ induced at $0$ from the map $\tau_{-q} \circ (\phi_q|_{\mathbb R^n}) \circ \tau_{p_0}$ belongs to $\mathcal G$. Here we denote by $\tau_p: \mathbb R^n\to \mathbb R^n$ the translation $\tau_p(x) = x + p$.
\end{lemma}
\begin{proof}
Let $p\in M$; we are first going to concentrate on the case $\mathcal G = {\rm Diff}^\omega(\mathbb R^n,0)$. The statement we want to prove is local, thus we we can just consider $M\cap B_1(p)$, where $B_1(p)$ is the ball of center $p$ and radius $1$. Let $\mathbb Q_+ = \{x\in \mathbb Q: x>0\}$. For any pair $(r,\epsilon)\in \mathbb Q_+^2$, we define
\[A_{r,\epsilon}=\{q\in B_1(p)\cap M: \exists \psi\in {\rm Hol}_*(B_r(p)) {\rm s.t.} |\psi-p|\leq 2\ {\rm and}\ |J\psi|\geq\epsilon\  {\rm on}\ B_r(p), \]
\[\psi_{|B_r(p)\cap \mathbb R^n}\in {\rm Diff}^\omega_*(B_r(p)\cap \mathbb R^n), \psi(p)=q, \psi(M)\subset M \}. \]

Then we have
\[\bigcup_{(r,\epsilon)\in \mathbb Q_+^2} A_{r,\epsilon} = B_1(p)\cap M;\]
indeed, let $q\in B_1(p)\cap M$, and let $\psi$ be a real-analytic diffeomorphism preserving $M$, defined in a neighborhood of $p$ in $\mathbb R^n$, such that $\psi(p)=q$. We can extend $\psi$ to a holomorphic mapping (still denoted by $\psi$) defined in some ball $B_r(p)$ in $\mathbb C^n$; since $|q-p|<1$, by shrinking $r$ we can achieve that $|\psi(z)-p|<2$ for $z\in B_r(p)$. Moreover, since the Jacobian $J\psi(p)$ is non-vanishing, by shrinking $r$ again if necessary we obtain that $|J\psi|$ is never vanishing on $\overline B_r(p)$ and, moreover, $\psi$ admits an inverse defined on $\psi(B_r(p))$. Taking $(r',\epsilon)\in \mathbb Q_+^2$ such that $r'<r$ and $\epsilon<\inf_{B_r(p)} |J\psi|$ we have that $q\in A_{r',\epsilon}$. 

We will now show that $A_{r,\epsilon}$ is a closed subset of $B_1(p)\cap M$ for any fixed $(r,\epsilon)\in \mathbb Q_+^2$. Let $\{q_j\}_{j\in \mathbb N}$ be a sequence contained in $A_{r,\epsilon}$ such that $q_j\to q\in B_1(p)\cap M$, and let $\{\psi_j\}_{j\in \mathbb N}$ be a sequence of mappings in ${\rm Hol}_*(B_r(p))$, satisfying the properties required in the definition of $A_{r,\epsilon}$, such that $\psi_j(p) = q_j$. Since $|\psi_j - p|\leq 2$ on $B_r(p)$ for all $j$, the sequence $\{\psi_j\}$ is uniformly bounded and by Montel's theorem we have (up to selecting a subsequence) that $\psi_j\to \psi \in {\rm Hol}(B_r(p))$ as $j\to \infty$, where the convergence is uniform on compact subsets of $B_r(p)$. By Lemma \ref{limitinv} we deduce that, in fact, $\psi\in {\rm Hol}_*(B_r(p))$. The fact that $\psi(p)=q$, as well as all the properties defining the elements of $A_{r,\epsilon}$, follow by continuity; we conclude that $q\in A_{r,\epsilon}$.

Applying Baire's theorem, we deduce that there exists $(r_0,\epsilon_0)\in \mathbb Q_+^2$ such that $A_{r_0,\epsilon_0}$ has non-empty interior; let $A'$ be an open subset contained in $A_{r_0,\epsilon_0}$, and for any $q\in A'$ denote by $\psi_q$ an element of ${\rm Hol}_*(B_{r_0}(p))$ such that $\psi_q(p)=q$ and satisfying the conditions defining $A_{r_0,\epsilon_0}$. Let $p_0\in A_{r_0,\epsilon_0}$; taking $U=\psi_{p_0}(B_{r_0}(p))$,
the family of mappings $\phi_q = \psi_q \circ \psi_{p_0}^{-1}$, $q\in A'\cap U$, fulfills the requirements of the lemma.

If $\mathcal G$ is a general closed subgroup of ${\rm Diff}^\omega(\mathbb R^n,0)$ we follow the same proof, the only change being that we add to the definition of $A_{r,\epsilon}$ the requirement that the germ of $\tau_{-q} \circ (\psi|_{\mathbb R^n})\circ \tau_p$ at $0$ belongs to $\mathcal G$. The fact that $\cup A_{r,\epsilon} = B_1(p)\cap M$ follows with the same argument as before, by applying the definition of $\mathcal G$-homogeneity. In the proof that $A_{r,\epsilon}$ is closed, considering the same sequence $\{\psi_j\}\subset {\rm Hol}_*(B_r(p))$ as above, we have that $\tau_{-q_j}\circ (\psi_j|_{\mathbb R^n})\circ \tau_p$ converge uniformly to $\tau_{-q}\circ(\psi|_{\mathbb R^n})\circ \tau_p$ on a fixed neighborhood of $0$ in $\mathbb R^n$. The fact that the germ of $\tau_{-q}\circ(\psi|_{\mathbb R^n})\circ \tau_p$ at $0$ belongs to $\mathcal G$ is then a direct consequence of the definition of closed subgroup of ${\rm Diff}^\omega(\mathbb R^n,0)$. Finally, at the conclusion of the proof we have that 
\[\tau_{-q}\circ(\phi_q|_{\mathbb R^n})\circ \tau_{p_0} = \left (\tau_{-q}\circ(\psi_q|_{\mathbb R^n})\circ  \tau_p \right ) \circ \left ( \tau_{-p}\circ(\psi^{-1}_{p_0}|_{\mathbb R^n})\circ \tau_{p_0} \right )\]
thus, since $\mathcal G$ is a subgroup, the germ at $0$ of $\tau_{-q}\circ(\phi_q|_{\mathbb R^n})\circ \tau_{p_0}$ belongs to $\mathcal G$.
\end{proof}

\subsection{Real-analyticity of certain $\mathcal G$-homogeneous curves}\label{proofpro}

From now on we are going to only consider the case of planar curves, that is we will assume that $M=\gamma\subset \mathbb R^2$ is a $1$-dimensional curve of class $C^\infty$ and we will use Theorem \ref{Baire} for $n=2$. Suppose that $\gamma$ is $\mathcal G$-homogeneous for some closed subgroup $\mathcal G\leq {\rm Diff}^\omega(\mathbb R^2,0)$. The results of the previous section show that it is possible to choose $p\in \gamma$, a neighborhood $U$ of $p$ in $\mathbb R^2$ and a family $\{\phi_q\}_{q\in \gamma\cap U}$ of real-analytic automorphisms of $\gamma$, all defined and uniformly bounded on $U$, such that $\phi_q(p) = q$. We have, however, no information about the dependence of the map $\phi_q$ on $q$, which needs not be smooth and might in principle be wildly discontinuous.

The purpose of this section is to show that, under a certain additional assumption on the group $\mathcal G$, the behavior of the family $\{\phi_q\}$ must be very rigid. This allows to apply the theory outlined in Section \ref{Cartan}, obtaining a non-trivial analytic vector field whose flow leaves $\gamma$ invariant, which in turn implies that $\gamma$ is real-analytic.

Let $\psi\in {\rm Diff}^\omega(\mathbb R^2,0)$, and let $\lambda$ be a $1$-dimensional curve of class $C^\infty$ around $0\in \mathbb R^2$, that is a closed $1$-dimensional submanifold of an open neighborhood of $0$ in $\mathbb R^2$ such that $0\in \gamma$. The curve $\lambda$ is called an \emph{invariant curve} for $\psi$ if $\psi(q) \in \lambda$ for any  $q\in \lambda$ close enough to $0$. Given, now, a subgroup $\mathcal G\leq {\rm Diff}^\omega(\mathbb R^2,0)$, we are interested in the following condition (we recall that we denote by $\mathcal R$ the set of germs of order 2, see section \ref{defnot}):
\begin{itemize}
\item[(ICP)] for every $g\in \mathcal G$, $g\not\in \mathcal R$, and for any curve $\lambda$ of class $C^\infty$ which is invariant under $g$ we have that $\lambda\setminus \{0\}$ is real-analytic, that is for any $p\in \lambda$, $p\neq 0$, there is a neighborhood $U$ of $p$ in $\mathbb R^2$ such that $\lambda\cap U$ is real-analytic.
\end{itemize}
which we call the \emph{invariant curve property}.

\begin{lemma} \label{onlyone}
Let $\mathcal G$ be a group satisfying (ICP), and let $\gamma$ be a nowhere analytic curve of class $C^\infty$, $0\in \gamma$. Then there exists at most one element $g\in \mathcal G\cap \mathcal R$, $g\neq Id$, such that $\gamma$ is $g$-invariant. 
\end{lemma}
\begin{proof}
Let $\phi\in \mathcal R$, and assume that $\gamma$ is $\phi$-invariant. Let $\mu=\pm 1$ be the eigenvalue of $d\phi(0)$ corresponding to the eigenvector $T_0(\gamma)$. We claim that either $\mu = - 1$ or $\phi=Id$. Note that if both the eigenvalues of $d\phi(0)$ are equal to $1$ then $\phi=Id$ (as the only element of $\mathcal R$ with this property is the identity). Thus if $\mu=1$ and $\phi\neq Id$ there exists an analytic change of coordinates such that $\phi(x,y) = (x,-y)$ and $T_0(\gamma) = \partial/\partial x$. It follows immediately that, in these coordinates, $\gamma = \{y=0\}$, contradicting the fact that $\gamma$ is nowhere analytic.

Suppose that there exist $Id\neq g_1,g_2\in \mathcal G\cap \mathcal R$ leaving $\gamma$ invariant and let $\phi = g_1\circ g_2 = g_1\circ g_2^{-1}$. Since $\gamma$ is $\phi$-invariant and nowhere analytic, by (ICP) we must have $\phi\in \mathcal G\cap \mathcal R$. On the other hand, by the first paragraph the eigenvalue of both $dg_1(0)$ and $dg_1(0)$ relative to the eigenvector $T_0(\gamma)$ is equal to $-1$, which implies that the corresponding eigenvalue for $d\phi(0)$ is $1$. Again by the first paragraph, it follows that $\phi = Id$, hence $g_1 = g_2$. 
\end{proof}

We are now in the position to provide a proof of Proposition \ref{Ghom}:
\begin{proof}[Proof of Proposition \ref{Ghom}]
Let $\{\phi_q\}$ be the family of holomorphic diffeomorphisms described in Lemma \ref{Baire} (where we can assume that $p_0 =0$), and suppose by contradiction that $\gamma$ is not real-analytic. By homogeneity, $\gamma$ must then be nowhere real-analytic. Applying Lemma \ref{onlyone}, we have that the subgroup of the germs $\psi\in\mathcal G$ such that $\gamma$ is $\psi$-invariant in a neighborhood of $0$ is either trivial or is of the form $\{g,Id\}$ for some $g\in \mathcal G\cap \mathcal R$. After possibly shrinking the open set $U$, the set of maps $H = \{\phi_q\} \cup \{\phi_q \circ g\}$ clearly satisfies all the properties of Lemma \ref{Baire} with the only exception that for $p\in \gamma\cap U$ there are two elements of $H$ mapping $0$ to $p$ instead of one. 

By a suitable affine transformation, we can assume that $U$ contains the ball $B_1(0)$ and restrict $H$ to the family $\{\phi_q\}_{q\in \gamma\cap B_1(0)}\cup \{\phi_q\circ g\}_{q\in \gamma\cap B_1(0)}$. By Lemma \ref{Baire}, the maps of $H$ are uniformly bounded by a constant $C>0$, hence by Montel's theorem $H$ fulfills property (a) (see the second paragraph of Section \ref{Cartan}). We can thus define neighborhoods of $Id$ in $H$ of the form $H_\epsilon = \{\phi\in H : \delta_{\overline B_{1/2}(0)}(\phi)\leq \epsilon \}$.

Our next step is to prove the following fact. Let  $\{q_i\}_{i\in \mathbb N}\subset \gamma$ with $q_i \to q_0\in \gamma$ and $\{\varphi_i\}_{i\in \mathbb N} \subset H$ such that $\varphi_i(0) = q_i$. Then any convergent subsequence of $\{\varphi_i\}$ (note that by Montel's theorem there is at least one) converges to an element of $H$. Indeed, let $\psi\in {\rm Hol}(B_1(0))$ be the limit of a convergent subsequence of $\{\varphi_i\}$: by continuity, $\psi(0) = q_0$ and $\psi(\gamma)\subset \gamma$. Furthermore, the fact that $\mathcal G$ is a closed subgroup of ${\rm Diff}^\omega(\mathbb R^2,0)$ implies that the germ of $\tau_{-q_0} \circ (\psi|_{\mathbb R^2})$ at $0$ belongs to $\mathcal G$. Let $V$ be a small neighborhood of $0$ in $\mathbb R^2$ such that $\psi(V) \subset \phi_{q_0}(B_1(0))$; we can thus define a map $\Psi\in {\rm Diff}^\omega(V)$ such that $\Psi(0)=0$ as $\Psi = \phi_{q_0}^{-1} \circ \psi$. Since the germ induced by $\Psi$ at $0$ belongs to $\mathcal G$, and $\gamma$ is nowhere analytic and $\Psi$-invariant, by (ICP) and Lemma \ref{onlyone} we have that this germ is either the identity or $g$. By analytic continuation we conclude that either $\psi \equiv \phi_{q_0}$ or $\psi \equiv  \phi_{q_0}\circ g$ on $B_1(0)\cap \mathbb R^2$ and hence on $B_1(0)$; in both cases we have $\psi\in H$.

We will now show that for $\epsilon_0< \delta_{\overline B_{1/2}(0)}(g)$ the set $H_{\epsilon_0}$ fulfills  the following property: there is a neighborhood $F$ of $0$ in $\gamma$ such that, for any $q\in F$, there exists exactly one $\varphi_q\in H_{\epsilon_0}$ such that $\varphi_q(0) = q$. Indeed, otherwise there would either exist a sequence $q_i\to 0$ such that $\phi_{q_i}$ and $\phi_{q_i}\circ g$ both lie in $H_{\epsilon_0}$, or a sequence $q_i\to 0$ such that none of them does. In the first case, taking a convergent subsequence as above we can find a sequence $\{\varphi_i\}\subset H_{\epsilon_0}$ such that either $\varphi_i\to g$ (this is an immediate contradiction since $g\not\in H_{\epsilon_0}$) or $\varphi_i\to Id$, but then we have $\{\varphi_i\circ g\}\subset H_{\epsilon_0}$ and $\varphi_i\circ g \to g$, again a contradiction. The second case leads to a contradiction by the same procedure.

For any $q\in F$, we now define $\varphi_q=\phi_q$ if $\phi_q\in H_{\epsilon_0}$ and $\varphi_q =  \phi_q\circ g$ otherwise. Set $G=\{\varphi_q\}_{q\in F}$. We claim that $G$ is a locally closed, local group of holomorphic transformations, in the sense defined in Section \ref{Cartan}. As before, we put $G_\varepsilon = \{\varphi\in G : \delta_{\overline B_{1/2}(0)}(\varphi)\leq \varepsilon \}$

We can verify that the map $\gamma\ni q \to \varphi_q \in G$ is continuous for $q$ close to $0$. Indeed, suppose it is not so, and let $q_0\in F$, $\{q_i\}_{i\in \mathbb N}\subset F$ and $c>0$ be chosen in such a way that $q_i \to q_0$ as $i\to \infty$ but $\delta_{B_{1/2}(0)}(\varphi_{q_i},\varphi_{q_0}) > c$  for $i\in \mathbb N$. As already shown, up to taking a subsequence we have that $\{\varphi_{q_i}\}$ converges uniformly on the compact subsets of $B_1(0)$ to a map $\psi\in H_{\epsilon_0}$ such that $\psi(0)=q_0$. By construction $\psi = \varphi_{q_0}$, which gives a contradiction. 

A completely analogous argument shows that $G$ is locally compact, and thus locally closed. Moreover, repeating the argument with $q_0 = 0$ we find a sequence $\{\varphi_i\}\subset G$ which  converges to $Id$, showing that $Id$ is not isolated in $G$. In order to see that $G$ is a local group of transformations we must now prove that $G$ satisfies (b) and (c).

We will check that $G$ satisfies property (b) with  $G' = G_{\kappa}$ with $\kappa< \min\{\frac{\epsilon_0}{2}, \frac{1}{4}{\rm diam}F \}$. Let $\varphi_{q_1}, \varphi_{q_2}\in G'$, and let $q_3=\varphi_{q_2}(q_1)$; note that $q_3\in F$ by definition of $G'$. Let $\varphi_{q_3}\in G$ be such that $\varphi_{q_3}(0) = q_3$ and consider the map $\Phi\in {\rm Hol}(\varphi_{q_1}^{-1}(B_1(0)))$ obtained as the composition $\Phi = \varphi_{q_2}\circ \varphi_{q_1}$. Then $\Phi(0) = q_3$ and the fact that $\varphi_{q_1}\in G'$ implies that $\Phi$ is defined (at least) over $B_{1/2}(0)$. Let $V$ be a neighborhood of $q_3$ in $\mathbb R^2$ which is contained $\varphi_{q_3}(B_1(0))$; then we can select a small enough neighborhood $W\subset B_{1/2}(0)$ of $0$ in $\mathbb R^2$ such that $\Phi(W)\subset V$. With such a choice, we have that the map $\varphi_{q_3}^{-1}\circ \Phi$ is defined on $W$ and $\varphi_{q_3}^{-1}\circ \Phi(0) = 0$. Moreover, by Lemma \ref{Baire} and by the assumption that $\mathcal G$ is a uniform subgroup of ${\rm Diff}^\omega(\mathbb R^2,0)$, the germ induced by $\varphi_{q_3}^{-1}\circ \Phi$ at $0$ belongs to $\mathcal G$ and the germ of $\gamma$ at $0$ is a smooth, nowhere analytic invariant curve for it. By (ICP) and Lemma \ref{onlyone} the germ of $\varphi_{q_3}^{-1}\circ \Phi$ is either $Id$ or $g$, i.e.\ $\Phi$ either coincides with $\varphi_{q_3}$ or with $\varphi_{q_3}\circ g$ on $W$ (and by analytic continuation, everywhere). The second case is not possible since by construction $\delta_{B_{1/2}}(\Phi)<\epsilon_0$ and $\delta_{B_{1/2}}(\varphi_{q_3}\circ g)>\epsilon_0$. Thus we have $\Phi\equiv \varphi_{q_3}$; in particular, $\Phi$ extends holomorphically to $B_1(0)$ and is bounded by $C$.


Choose, now, $\Delta\Subset B_1(0)$ as required by condition (b), and select $\eta(\Delta)<{\rm dist}(\overline\Delta, bB_1(0))$ such that $\delta_{B_{1/2}}(\phi)<\kappa$ for all $\varphi\in G$ for which $\delta_{\Delta}(\phi)<\eta(\Delta)$. This choice is possible because $G$ satisfies condition (a), and immediately implies that (b)-(1) and (b)-(2) are fulfilled for $\eta(\Delta)$. To verify (b)-(3), pick $\varphi_{q_1}\in G$ such that $\delta_{\Delta}(\varphi_{q_1})<\eta(\Delta)$ and $\varphi_{q_2}\in G'$. Then the composition of $\varphi_{q_1}$ and $\varphi_{q_2}$ requested in (b)-(3) is given by $\varphi_{q_3}\in G$, with $q_3 = \varphi_{q_2}(q_1)$. By the discussion of the paragraph above, $\varphi_{q_3}$ is the analytic continuation of the map $\Phi = \varphi_{q_2} \circ \varphi_{q_1}$, which is a priori defined on $\varphi_{q_1}^{-1}(B_1(0))$. On the other hand, by construction $\Delta \subset \varphi_{q_1}^{-1}(B_1(0))$, hence for all $z\in \Delta$ we have $\varphi_{q_3}(z) = \Phi(z) = \varphi_{q_2}(\varphi_{q_1}(z))$. This shows that also (b)-(3) is fulfilled.

The verification of (c) is quite similar. We start by choosing a ball $B_r(0)$, with the radius $r<1/2$ small enough, so that the following points hold:
\begin{itemize}
\item $bB_{r}(0)\cap \gamma$ consists of exactly two points $p_1,p_2$,  belonging to different connected components $\gamma'$, $\gamma''$ of $\gamma\setminus \{0\}$;
\item  $B_{r/2}(p_1)\cap \gamma \subset \gamma'$, $B_{r/2}(p_2)\cap \gammaÊ\subset \gamma''$, and $[p_1,p_2]\subset \overline B_{r}(0)$ where $[p_1,p_2]$ is the segment of $\gamma$ comprised between $p_1$ and $p_2$;
\item $\varphi_q\in G'$ for all $q\in [p_1,p_2]$.
\end{itemize}
The choice of $r>0$ satisfying the first two condition is possible because of the smoothness of $\gamma$. The third point is verified for $r$ small enough because of continuity of the map $q\to \varphi_q$. With this choice of $r$, we define the neighborhood $G''\subset G'$ required by condition (c) as $G'' = G_{r/2}$. Then, given $\varphi\in G''$ we have $\varphi(p_1)\in \gamma'$, $\varphi(p_2)\in \gamma''$, from which follows that $\varphi([p_1,p_2])\subset\gamma$ is a connected subset intersecting both $\gamma'$ and $\gamma''$, hence $0\in \varphi([p_1,p_2])$. (An alternative way of checking that $0\in \phi(B_1(0))$ for $\phi\in G''$ is by applying the several-variables version of the Rouch\'e theorem \cite{Ll}).

Let, then, $\varphi_{q_1}\in G''$;  by the arguments above we have that $q_2 = \varphi_{q_1}^{-1}(0) \in B_r(0)\cap \gamma$. Let $\varphi_{q_2}\in G$ such that $\varphi_{q_2}(0) = q_2\in [p_1,p_2]$; by construction, this implies in particular $\varphi_{q_2} \in G'$. The fact that $\varphi_{q_1}^{-1}$ coincides with $\varphi_{q_2}$ follows from (ICP) and Lemma \ref{onlyone} by considering the map $\Phi= \varphi_{q_1}\circ \varphi_{q_2}$ (defined on a neighborhood of $0$ in $\mathbb R^2$) and applying the same argument used in the verification of (b). This concludes the proof of the claim.

We are now in the position to apply Corollary \ref{puttog} to the local group $G$. We obtain a non-trivial holomorphic vector field $\Psi = \sum_j \psi_j \partial / \partial z_j$ on $B_1(0)$ such that, denoting by $\psi_t(z) = \psi(z,t)$ the flow of ${\rm Re}\Psi$, we have $\psi_t\in G$ for all $t$ in a neighborhood of $0$ in $\mathbb R$. We observe here that, defined $\psi = (\psi_1,\psi_2)$, $\psi|_{\mathbb R^2}$ is valued in $\mathbb R^2$; this is true because $\psi$ is obtained as the limit of maps of the kind $m_j(\varphi_{q_j}-Id )$ (see Theorem \ref{Th3}). In particular the fact that $\psi$ does not vanish identically implies that ${\rm Re}\psi|_{\mathbb R^2} = ({\rm Re}\psi_1|_{\mathbb R^2}, {\rm Re}\psi_2|_{\mathbb R^2})= (\psi_1|_{\mathbb R^2}, \psi_2|_{\mathbb R^2})$ does not vanish identically, because the latter would imply $\psi|_{\mathbb R^2}\equiv 0$. We deduce that the vector field ${\rm Re}\Psi|_{\mathbb R^2}$ is a non-identically-vanishing, real analytic vector field on $\mathbb R^2\cap B_1(0)$.

We claim that $0\not\in \{ {\rm Re}\Psi|_{\mathbb R^2}=0 \}$. Indeed, that would imply that $\psi_t(0) = 0$ for all small $t\in \mathbb R$, but since $\psi_t\in G$ we would have $\psi_t = \varphi_0 = Id$ for all $t$, which contradicts the fact that ${\rm Re}\Psi|_{\mathbb R^2}$ is not identically vanishing. Since $0$ belongs to the set where ${\rm Re}\Psi|_{\mathbb R^2}$ is non-singular, from the fact that $\psi_t\in G$ -- hence $\psi_t(0)\in \gamma$ -- for all small enough $t\in \mathbb R$ we deduce that $\gamma$ locally coincides with the flow line of ${\rm Re}\Psi|_{\mathbb R^2}$ through $0$. It follows that $\gamma$ is real-analytic.
\end{proof}

\begin{remark}
In the previous proposition, we can replace the assumption \lq\lq $\gamma$ is $\mathcal G$-homogeneous, with $\mathcal G$ satisfying (ICP)\rq\rq\ with an assumption regarding the curve alone: \lq\lq $\gamma$ is ${\rm Diff}^\omega(\mathbb R^2,0)$-homogeneous and its only analytic isotropies are involutions\rq\rq\ (i.e.\ any analytic germ $\psi$ such that $\psi(0)=0$ which leaves $\gamma$ invariant belongs to $\mathcal R$). In this case, the argument of Lemma \ref{onlyone} still implies that the isotropy group of $\gamma$ (should it be nowhere analytic) consists of a single element of $\mathcal R$, and this is the only fact (together with Lemma \ref{Baire}, which is valid for the whole group ${\rm Diff}^\omega(\mathbb R^2,0)$) which is actually needed in the proof of Proposition \ref{Ghom}.
\end{remark}

\subsection{Analyticity of invariant curves outside the origin}\label{aninvcur}

In this section, we are going to show that the property (ICP) is satisfied, in various forms, for some interesting subgroups $\mathcal G\leq {\rm Diff}^\omega(\mathbb R^2,0)$. In connection with Proposition \ref{Ghom}, this shows that for these subgroups $\mathcal G$ the curves which are $\mathcal G$-homogeneous are real-analytic.

We start with the case $\mathcal G = {\rm Hol}(\mathbb C,0)$ as mentioned in Lemma \ref{leaufatou}:


\begin{proof}[Proof of Lemma \ref{leaufatou}]
Write the expansion of $f$ around $0$ as $f(z) = \lambda z + O(z^2)$, $\lambda\in \mathbb C$. Since the differential of $f$ at $0$ must preserve the direction $T_0(\gamma)$, it follows that $\lambda\in \mathbb R\setminus \{0\}$. Up to replacing $f$ with $f^{\circ 2} = f\circ f$ (which is not the identity since $f\not \in \mathcal R$), we can assume $\lambda >0$, and possibly considering $f^{-1}$ instead of $f$ we can further suppose $\lambda \leq 1$.

If $0<\lambda< 1$, then $f$ is holomorphically conjugated to its linearization $\widetilde f(z) = \lambda z$ (see for example \cite[Theorem 2.1]{Br}); let $\widetilde \gamma$ be the image of $\gamma$ under the linearizing change of coordinates. Up to a rotation, we can assume that $T_0(\widetilde \gamma)$ is horizontal (i.e.\ generated by $\partial/\partial x$). Then $\widetilde \gamma$ coincides with the $x$-axis; otherwise, choose $p\in \widetilde \gamma$ such that $\arg T_p(\widetilde \gamma)\neq 0$ (here and in the rest of the proof, we are going to improperly apply the function $\arg$ to linear subspaces $T\subset\mathbb C$, defined modulo multiples of $\pi$ as the argument of a vector generating $T$). By invariance of $\widetilde \gamma$ under  $\widetilde f$  we have that $\arg T_{p_j}(\widetilde \gamma)=\arg T_p(\widetilde \gamma)\neq 0$ for all $j\in \mathbb N$, where $p_j= \widetilde f^{\circ j}(p) = \lambda^jp$: it follows that $\widetilde \gamma$ is not of class $C^1$ at $0$, a contradiction. Therefore, in this case the curve $\widetilde\gamma$ (and hence $\gamma$) is actually real-analytic around $0$.

To treat the case when $\lambda =1$ (the \emph{parabolic} case) we will use the Leau-Fatou flower theorem, which provides a description of the dynamics of such a germ $f$. Since we need to examine the proof of this result rather than its statement alone, we shall refer to the proof which is contained in \cite[Theorem 2.12]{Br}, and employ the notation set up in there.

We first recall the essential features of the theorem. We can write the Taylor expansion of $f$ as $f(z) = z + a_{k+1}z^{k+1} + O(z^{k+2})$ with $k\geq 1$, $a_{k+1}\in \mathbb C\setminus\{0\}$. The $k$ directions $v^+_1,\ldots,v^+_k\in bD \cong S^1$ which solve the equation $\frac{a_{k+1}}{|a_{k+1}|} v^k= -1$ are called \emph{attracting directions} for $f$. Then there exist simply connected domains $P_{v^+_1},\ldots, P_{v^+_k}\subset \mathbb C$ with the following properties:
\begin{itemize}
\item[(1)] $0\in bP_{v_j^+}$ and $f(P_{v_j^+})\subset P_{v_j^+}$;
\item[(2)] $\lim_{n\to \infty} f^{\circ n}(z) = 0$ and $\lim_{n\to \infty} \frac{f^{\circ n}(z)}{|f^{\circ n}(z)|} = v^+_j$ for all $z\in P_{v_j^+}$.
\end{itemize}
The domain $P_{v_j^+}$ is called \emph{attracting petal} centered at the direction $v_j^+$. The \emph{repelling directions} $v_j^-$ and the \emph{repelling petals} $P_{v_j^-}$ are the attracting directions/petals associated to the germ $f^{-1}$. The attracting and repelling petals can be so chosen that their union (plus the point $0$) is a neighborhood of $0$ in $\mathbb C$; also, from the proof follows that each petal locally contains an open sector centered at $0$. Moreover, for any $j=1,\ldots,k$ the action $f_{| P_{v^+_j}}$ is holomorphically conjugated to the map $\zeta \to \zeta + 1$, defined on a half-plane of the form $\{\zeta\in \mathbb C : {\rm Re}\zeta >C \}$ for some $C>0$.

Let $P$ be a petal which intersects $\gamma$; without loss of generality (possibly considering $f^{-1}$ in place of $f$ and conjugating with a complex linear transformation) we can suppose that $P$ is an attracting petal, centered at $v=1$. From the property (2) above we deduce that $T_0(\gamma)$ is the $x$-axis; it also follows that $\gamma\cap \{x>0\}$ is locally contained in $P$.

Let $\Psi$ be a map conjugating $f$ to $\zeta \to \zeta + 1$ (such a $\Psi$ is called \emph{Fatou coordinate}), and let $\widetilde \gamma$ be the image of $\gamma$ under $\Psi$. Our aim is to show that $\widetilde\gamma$ is of the form $\{y=y_0\}$ for some $y_0\in \mathbb R$. If $\widetilde \gamma$ does not coincide with a horizontal line, there exists $p\in \widetilde\gamma$ such that $\arg T_p(\widetilde\gamma) = \alpha \neq 0$. This of course also implies $\arg T_{p+n}(\widetilde\gamma) = \alpha$ for all $n\in \mathbb N$.

We are thus lead to computing the differential $d(\Psi^{-1})$ at the point $p+n$, which is given by the multiplication by a certain $\xi_n\in \mathbb C\setminus\{0\}$. We are going to show that $\arg \xi_n \to \pi$ as $n\to \infty$: posing $q_n = \Psi^{-1}(p+n)$, this would imply that $q_n\to 0$ and $\arg T_{q_n}(\gamma) = \arg \xi_n  + \alpha \to \pi + \alpha \neq \{0,\pi\}$ as $n\to \infty$, which would contradict the fact that $\gamma$ is of class $C^1$.

We turn then to the construction of the Fatou coordinate $\Psi$; as mentioned above, we follow the one given in \cite{Br}. The map $\Psi$ is obtained in two steps. First of all, the restriction of $f$ to a domain of the kind $\{|z^k - \delta| < \delta\}$ (for a small $\delta>0$) is conjugated, through the map $\psi(z) = 1/(kz^k)$, to a function $\varphi:H_\delta\to H_\delta$ of the kind 
\begin{equation}\label{firstcon}
\varphi(z) = z + 1 + b/z + R(z)
\end{equation}
 where $R(z) = O(1/z^2)$ and $H_\delta = \{ {\rm Re} w >1/2k\delta\}$. Afterwards, $\varphi$ is conjugated to the translation $\zetaÊ\to \zeta + 1$ through a holomorphic mapping $\sigma: H_\delta \to \mathbb C$ (that is to say, $\sigma\circ \varphi(z) = \sigma(z) + 1$), so that $\Psi = \sigma \circ \psi$.

Let $p\in \widetilde\gamma$ be the point chosen above. We set $r = \sigma^{-1}(p)$ and, for $n\in \mathbb N$, $r_n = \sigma^{-1}(p+n)$. It also follows that $r_n = \psi(q_n)$ (where the points $q_n=\Psi^{-1}(p+n)$ are defined above) and that $r_n = \varphi^{\circ n}(r)$. We have that $\arg r_n \to 0$ as $n\to \infty$. Indeed, for every $z\in H_\delta$ one has (see \cite[Th. 2.12, Eq. (2.18)]{Br}) $|\varphi^{\circ n}(z)|= O(n)$: by (\ref{firstcon}), this implies $r_{n+1}- r_n = 1 + O(1/n)$. It follows that for any $\epsilon>0$ there exists $n_0\in \mathbb N$ such that $|\arg (r_n - r_{n_0})| < \epsilon$ for all $n>n_0$ (in fact, for all large $n$, $r_{n+1}$ is contained in a sector centered at $r_n$ with opening angle less than $\epsilon$), which in turn implies that $|\arg r_n| < 2\epsilon$ for all large enough $n$, as claimed.

Computing, now, the derivative of $\psi^{-1}(z) = 1/(kz)^{1/k}$ gives $\frac{\partial}{\partial z}\psi^{-1}(z) = - 1/(kz)^{\frac{k+1}{k}}$. It follows that $\arg (\frac{\partial}{\partial z}\psi^{-1}(r_n)) = \pi - \frac{k+1}{k}\arg r_n \to \pi$ as $n\to \infty$. Since $\Psi^{-1} = \psi^{-1} \circ \sigma^{-1}$, with $\xi_n$ as previously defined we get $\xi_n = \frac{\partial}{\partial z}\psi^{-1}(r_n) \cdot \frac{\partial}{\partial z}\sigma^{-1}(p+n)$, thus to show that $\arg \xi_n \to \pi$ it is sufficient to show that $ \frac{\partial}{\partial z}\sigma^{-1}(p+n)\to 1$ or, equivalently, that $\frac{\partial}{\partial z}\sigma(r_n)\to 1$ as $n\to \infty$. We will concentrate on the latter.

The mapping $\sigma$ is constructed as the limit of the functions $\sigma_n(z) = \varphi^{\circ n}(z) - n - b\log n$; it can be shown that the sequence $\{\sigma_n\}_{n\in \mathbb N}$ is uniformly convergent on compact subsets of $H_\delta$. We will prove that, for any $\epsilon>0$, we can fix a sufficient large $n_0$ such that $|\frac{\partial}{\partial z}\sigma_j(r_n)-1| = |\frac{\partial}{\partial z}\varphi^{\circ j}(r_n) - 1| <\epsilon$ for all $n\geq n_0$ and $j\in \mathbb N$. This will imply that $|\frac{\partial}{\partial z}\sigma(r_n) - 1|<\epsilon$ for $n\geq n_0$, which is the desired conclusion.

In order to do this, we need to estimate the derivatives of $\varphi^{\circ j}$. Let $R(z)$ be the function appearing in the expression (\ref{firstcon}); then $R$ is not obtained as a convergent series in $1/z$ (indeed, one should not expect $\varphi$ to extend meromorphically to a neighborhood of $\infty$). However, from the computation performed in the proof of the Leau-Fatou theorem follows that there is a convergent series $S\in \mathbb C\{x\}$, $S(x)=\sum_{i\geq 2k} s_i x^i$, such that $R(z) = S(1/z^{\frac{1}{k}})$. Since $S'(x)\leq C_0|x|^{2k-1}$ for some $C_0>0$ and 
\[\frac{\partial}{\partial z}R(z) =  S' \left(\frac{1}{z^{\frac{1}{k}}} \right)\cdot \left(- \frac{1}{z^{\frac{k+1}{k}}}\right),\]
we get $| \frac{\partial}{\partial z}R(z)| \leq C_0 /|z|^{ \frac{2k-1}{k} + \frac{k+1}{k}} = C_0 /|z|^{3}$. Posing $T(z) = -b/z^2 + \frac{\partial}{\partial z}R(z)$, we deduce that $|T(z)|\leq C_1/|z|^2$ for some $C_1>0$ and $\frac{\partial}{\partial z} \varphi(z) = 1 + T(z)$. Now from (\ref{firstcon}) we get, for all $z\in H_\delta$,
\[\varphi^{\circ (j+1)}(z) = \varphi^{\circ j}(z) + 1 + \frac{b}{\varphi^{\circ j}(z)} + R(\varphi^{\circ j}(z)) \]
differentiating which we obtain
\[ \frac{\partial}{\partial z}\varphi^{\circ(j+1)}(z) - \frac{\partial}{\partial z}\varphi^{\circ j}(z) = T(\varphi^{\circ j}(z))\cdot \frac{\partial}{\partial z}\varphi^{\circ j}(z)  \]
so that
\begin{equation}\label{sothat}
\left | \frac{\partial}{\partial z}\varphi^{\circ(j+1)}(z) - \frac{\partial}{\partial z}\varphi^{\circ j}(z)\right | \leq \frac{C_1}{|\varphi^{\circ j}(z)|^2}\cdot \left |\frac{\partial}{\partial z}\varphi^{\circ j}(z)\right |.
\end{equation}
Fix any small $\epsilon>0$. Then we can select $n_1\in \mathbb N$ such that $\sum_{n=n_1}^\infty  1/n^2 <\epsilon/8C_1$. Choose a point $z_1\in H_\delta$ such that ${\rm Re \,}z_1\geq n_1/2$. Again from the proof of the Leau-Fatou theorem in \cite{Br} (see Th. 2.12, Eq. (2.15)), we get that 
\begin{equation}\label{again} 
{\rm Re \,}\varphi^{\circ j}(z) >  {\rm Re \,}z + j/2 \ \ \ (\Rightarrow |\varphi^{\circ j}(z)| > {\rm Re \,}z + j/2)
\end{equation}
for all $z\in H_\delta$, $j\in\mathbb N$. We will now prove by induction that $|\frac{\partial}{\partial z}\varphi^{\circ j}(z_1) - 1|< \epsilon$ for all $j\in \mathbb N$. For $j=1$, by definition of $T$ we have $\frac{\partial}{\partial z}\varphi(z_1) - 1 = T(z_1)$, hence $|\frac{\partial}{\partial z}\varphi(z_1) - 1| = |T(z_1)| \leq C_1/|z_1|^2 \leq 4C_1/(n_1)^2 < \epsilon/2$. Suppose, then, that for some $j\in \mathbb N$ we have $|\frac{\partial}{\partial z}\varphi^{\circ j}(z_1) - 1| \leq 8 C_1\sum_{i=n_1}^{n_1 + j-1} 1/i^2 < \epsilon$; in particular, this implies $|\frac{\partial}{\partial z}\varphi^{\circ j}(z_1)|<2$. Using (\ref{sothat}) and (\ref{again}), we get
\[\left | \frac{\partial}{\partial z}\varphi^{\circ(j+1)}(z_1) -  1\right | \leq \left | \frac{\partial}{\partial z}\varphi^{\circ(j+1)}(z_1) - \frac{\partial}{\partial z}\varphi^{\circ j}(z_1)\right |  + \left | \frac{\partial}{\partial z}\varphi^{\circ j}(z_1) - 1\right | \leq\]
\[\leq  \frac{4C_1}{(n_1+j)^2}\cdot 2 +  8 C_1\sum_{i=n_1}^{n_1 + j-1} \frac{1}{i^2} = 8C_1 \sum_{i=n_1}^{n_1 + j} \frac{1}{i^2}< \epsilon,\]
which gives the inductive step. Summing up, the previous argument provides the estimate $|\frac{\partial}{\partial z}\sigma(z_1) - 1|<\epsilon$ for all $z_1$ satisfying ${\rm Re \,}z_1\geq n_1/2$. On the other hand, by (\ref{again}) follows immediately that ${\rm Re \,}r_n = {\rm Re \,}\varphi^{\circ n}(r) \to \infty$ as $n\to \infty$. Together with the previous statement, this implies $\frac{\partial}{\partial z} \sigma(r_n) \to 1$ as $n\to \infty$, which concludes the proof of the lemma.
\end{proof}

Combining Lemma \ref{leaufatou} and Proposition \ref{Ghom}, we obtain Theorem \ref{holhomo}.


\

Next, we deal with $\mathcal G = {\rm Shr}^\omega(\mathbb R^2,0)$. In this case, we actually prove a slightly weaker form of (ICP) (for curves $\gamma$ which are a graph over $\mathbb R_x$), see Lemma \ref{invshr}. However, this is of course still sufficient to prove Proposition \ref{Ghom} -- if $\gamma$ is nowhere a graph over $\mathbb R_x$, then it must be of the form $\gamma=\{x=c\}$, hence it is real-analytic anyway.

\begin{proof}[Proof of Lemma \ref{invshr}]
The shear $\phi$ has the following form:
\[\phi(x,y) = (h(x), y + g(x))\]
where $h, g$ are germs of functions real-analytic around $0$ in $\mathbb R$ and $d h(0)/dx \neq 0$. We can write $h(x) = \lambda x + O(x^2)$, where $\lambda\in \mathbb R \setminus \{0\}$. Replacing $\phi$ by $\phi^{\circ 2}$ or $\phi^{\circ (-2)}$ (none of them is the  identity since $\phi\not \in \mathcal R$) if necessary, we can further assume that $0<\lambda\leq 1$.

Suppose, first, that $0<\lambda <1$.  The local holomorphic extension $h(z)$ of $h$ to a neighborhood of $0$ in $\mathbb C$ has the form $h(z)=\lambda z + O(z^2)$, hence it is linearizable by a local holomorphic change of coordinates $\eta(z)$. Since the restriction $h(x)$ of $h(z)$ to the real axis is real-valued, the same holds for the restriction $\eta(x)$ of $\eta(z)$ (cfr. \cite[Proposition 1.9]{Br}, where the coefficients of the series defining $\eta$ are explicitly computed). Conjugating $\phi$ by the map $(x,y)\to (\eta(x),y)$, it follows that in the new coordinates we can assume it to take the form $\phi(x,y)=(\lambda x, y + g_1(x))$ for a certain $g_1\in C^\omega(\mathbb R,0)$. Let $\{y=f_1(x)\}$ be the expression of $\gamma$ in these coordinates. The invariance of $\gamma$ under $\phi$, then, translates into the following identity
\begin{equation}\label{translates}
f_1(x) + g_1(x) = f_1(\lambda x),
\end{equation}
holding for $x$ in a small enough neighborhood of $0$ in $\mathbb R$. We can show by a direct power series computation that (\ref{translates}) admits, locally, a real-analytic solution $\widetilde f_1$. Indeed, if $g_1(x)=\sum_{j=1}^\infty a_j x^j$, looking for $\widetilde f_1$ of the form $\widetilde f_1(x) = \sum_{j=1}^\infty b_j x^j$ we get
\[ \sum_{j=1}^\infty (1 - \lambda^j)b_j x^j= -\sum_{j=1}^\infty a_j x^j\]
which has the unique solution $b_j = - a_j/(1-\lambda^j)$, $j\in \mathbb N$. From $0<\lambda<1$ follows that the factor $(1-\lambda^j)^{-1}$ is uniformly bounded in $j$, thus the series $\widetilde f_1$ has a positive radius of convergence. Let $\widetilde \gamma$ be the germ of real-analytic curve defined by $\{y = \widetilde f_1(x)\}$; we claim that $\gamma = \widetilde \gamma$.

In order to verify the claim, fix $C>0$ such that $|g_1(x)|\leq C|x|$ in a neighborhood of $0$, and let $p_0\in \gamma$, $p_0=(x_0,y_0)$. Then, if $\{p_j\}_{j\in \mathbb N}$ is the orbit of $p_0$ under $\phi$ (i.e.\ $p_j = \phi^{\circ j}(p_0)$) we also have $\{p_j\}\subset \gamma$ by invariance. One verifies by induction that 
\[p_j=(x_j,y_j) = (\lambda^j x_0, y_0 + \sum_{k=0}^{j-1} g_1(\lambda^k x_0))\] 
for all $j\geq 1$. Now, since $p_j\in \gamma$ and $x_j\to 0$ as $j\to \infty$, from the fact that $0\in \gamma$ follows that we must also have $y_j\to 0$ as $j\to \infty$. It follows that $\Sigma(x_0)=\sum_{k=0}^\infty g_1(\lambda^k x_0)$ converges and that $y_0 = -\Sigma(x_0)$ is uniquely determined by the abscissa $x_0$ and by the components of the shear $\phi$. We deduce that $f_1(x) = -\Sigma(x)$ so that $\gamma$ is in turn uniquely determined, hence $\gamma=\widetilde \gamma$ is real-analytic. We also observe that, since $\sum_k |g_1(\lambda ^k x)| \leq C|x| (1-\lambda)^{-1}$, the series defining $\Sigma(x)$ is in fact absolutely convergent; one can also check by a straightforward power series computation that $\sum_k g_1(\lambda^k x) = \sum_j b_j x^j$ (with $b_j$ as above).

We turn now to the case $\lambda = 1$. We note that we cannot have $h(x)\equiv x$; otherwise, in view of (\ref{translates}) (with $\lambda = 1$) we would also get $g(x)\equiv 0$, against the assumption that $\phi$ is a non-trivial germ. After a linear change of coordinates, and possibly taking $\phi^{-1}$ in place of $\phi$, we can thus assume that $h$ has the expression $h(x) = x - x^{k+1} + O(x^{k+2})$ for some $k\geq 1$. A further real-analytic conjugation of $h$ allows to put it in the form $h(x) = x - x^{k+1} + ax^{2k+1} + O(x^{2k+2})$ (see \cite[Remark 1.14]{Br}). The equation expressing the invariance of $\gamma$ under $\phi$ now reads
\begin{equation}\label{reads}
f(x) + g(x) = f(h(x)) = f(x + O(x^{k+1})).
\end{equation}
Denote by $g(w)$ the holomorphic extension of $g$ to a neighborhood $U$ of $0$ in $\mathbb C_w$, and let $g(w) = g_\ell w^\ell + O(w^{\ell +1})$, $\ell\geq 1$ be the Taylor expansion of $g$. Note that, in the case when $f$ is of class $C^\infty$, taking the $k$-order Taylor expansion about $0$ of both sides in (\ref{reads}) we get that $\ell\geq k+1$; we will show that the same conclusion can be drawn if $f$ is just assumed to be continuous.
 For a large enough $C>0$, we have $\frac{1}{C}|w|^\ell\leq |g(w)|\leq C|w|^{\ell}$ for all $w\in U$.

As before, let $h(z) = z - z^{k+1} + az^{2k+1} + O(z^{2k+2})$ be the local holomorphic extension of $h$ to a neighborhood of $0$ in $\mathbb C_z$, $z = x + iu$. In what follows we recycle the terminology and the notation employed in Lemma \ref{leaufatou}. Since the coefficient of $z^{k+1}$ is $-1$ we have that $v=1\in S^1$  is an attracting direction for the parabolic germ $h$. By the Leau-Fatou flower theorem, the positive  $x$-axis is the center of (hence locally contained in) an attracting petal $P\subset \mathbb C$. Consider the map $\psi(z) = 1/k z^k$, conjugating $h|_P$ to a function $\varphi:H_\delta\to H_\delta$ of the kind $\varphi(z) = z + 1 + b/z + R(z)$. We recall, from equation (\ref{again}), that  $|\varphi^{\circ j}(z)| > {\rm Re \,}z + j/2 > j/2$ for all $z\in H_\delta$, $j\in \mathbb N$.

We can choose a small enough $R>0$ such that $h^{\circ j}(z)\in U$ for all $z\in B_R(0)\cap P$ and $j\in \mathbb N$. Let $B_r(x)$ (for small $x,r>0$) be a ball contained in $B_R(0)\cap P$; then $\psi(B_r(x))\subset H_\delta$, which in view of the previous paragraph implies $|\varphi^{\circ j}(\psi(z))| >  j/2$ for all $z\in B_r(x)$, $j\in \mathbb N$. Composing with the inverse of $\psi$ we get
\begin{equation}\label{withinv}
|h^{\circ j}(z)| = |\psi^{-1} \circ \varphi^{\circ j}\circ \psi(z)| = \frac{1}{(k |\varphi^{\circ j}(\psi(z))| )^{\frac{1}{k}}} \leq \frac{D}{j^{\frac{1}{k}}}
\end{equation}
for all $z\in B_r(x)$, $j\in \mathbb N$, where $D = (2/k)^{1/k}$. On the other hand, for any $z\in H_\delta$ we have $|\varphi^{\circ j}(z)| = O(j)$ (see Lemma \ref{leaufatou}), so that with the same argument we get $h^{\circ j}(x') \geq D' / j^{1/k}$  for all small enough $x'\in \mathbb R^+$  and $j\in \mathbb N$, where $D'>0$ is a constant depending on $x'$ -- note that $h^{\circ j}(x')>0$ for small $x'>0$. 

Suppose first that $\ell\leq k$; by the choice of the constant $C$ above, we get for any small $x'>0$ $|g(h^{\circ j}(x'))| \geq D'^{\ell} / C j^{\frac{\ell}{k}} \geq D'' / j$. Moreover, the sign of $g(h^{\circ j}(x'))$ is constant for $j\in \mathbb N$, depending only on the sign of $g_\ell$. It follows that in the case $\ell \leq k$ the series $\sum_{j=0}^\infty g(h^{\circ j}(x'))$ is divergent for any small $x'>0$.

If instead $\ell \geq k+1$, by (\ref{withinv}) and the choice of $C$ we have that $|g(h^{\circ j}(z))| \leq CD^{k+1} / j^{\frac{k+1}{k}}$ for all $z\in B_r(x)$, $j\in \mathbb N$. It follows that the series $\sum_{j=0}^\infty g(h^{\circ j}(z))$ converges uniformly over $B_r(x)$ to a holomorphic function $\Sigma(z)$. Since for any $x>0$ small enough there exists $r>0$ such that $B_r(x)\subset B_R(0)\cap P$, we conclude that in the case $\ell \geq k+1$ the series $\Sigma(x) = \sum_{j=0}^\infty g(h^{\circ j}(x))$ converges and defines a real-analytic function on a neighborhood of $0$ in $\mathbb R^+$.

We will now show that $\gamma\cap \{x>0\}$ is real-analytic in a neighborhood of $0$ (the treatment of $\gamma\cap \{x<0\}$ is similar). Fix $p_0=(x_0,y_0)\in \gamma$ with $x_0>0$ small enough, and let $\{p_j = \phi^{\circ j}(p_0)\}$ be the orbit of $p_0$ under the shear map $\phi$. In the same way as before, we can inductively compute
\[p_j=(x_j,y_j) = (h^{\circ j}(x_0), y_0 + \sum_{k=0}^{j-1} g(h^{\circ k} (x_0)) \]
for all $j\in \mathbb N$. By the Leau-Fatou theorem, since $x_0$ belongs to the attracting petal $P$ for $h(z)$, we have $x_j\to 0$ as $j\to \infty$; since $\{p_j\}\subset \gamma$ and $0\in \gamma$, we must again have $y_j\to 0$. It follows that the series $\sum_{k=0}^\infty g(h^{\circ k} (x_0))$ converges -- hence, by the discussion above, $\ell \geq k+1$ -- and that $y_0 = -  \sum_{k=0}^\infty g(h^{\circ k} (x_0)) = - \Sigma(x_0)$. In conclusion, we have that $f(x) = -\Sigma(x)$ is real-analytic for $x>0$, hence $\gamma \cap \{x>0\}$ is real-analytic and, furthermore, it is univocally determined by the germ $\phi$ (since the series defining $\Sigma(x)$ only depends on $g$ and $h$).
\end{proof}

\begin{remark} In general, even when a shear $\psi\in {\rm Shr}^\omega(\mathbb R^2,0)$ admits a (unique) invariant curve, this needs not be real-analytic around $0$. For instance, defining
\[ \psi(x,y) = \left( \frac{x}{1-x}, y + x^2 \right)\]
then one can verify that $\psi$ admits an (at least continuous) invariant curve $\gamma$, but $\gamma$ is not real-analytic -- although $\gamma\setminus \{0\}$ is. Indeed, following the proof of Lemma \ref{invshr} we can define $f(0) = 0$ and
\[f(x) = - \sum_{k=0}^\infty \left ( \frac{x}{1 - kx}\right )^2  {\rm for} \ x<0, \ f(x) = \sum_{k=1}^\infty \left ( \frac{x}{1 + kx}\right )^2 {\rm for} \ x>0 \]
so that $\gamma = \{y = f(x)\}$ is the unique invariant curve for $\psi$. Clearly $\gamma\setminus \{0\}$ is real-analytic, but a straightforward computation shows that $f$ is not of class $C^2$ around $0$. One can also check that $\psi$ admits a unique formal invariant curve of the form $\{y = \widehat f(x)\}$ with $\widehat f\in \mathbb R[[x]]$: it follows that $\widehat f$ cannot be convergent, otherwise by Lemma \ref{invshr} its sum would locally coincide with $f(x)$. As it turns out, the coefficients of $\widehat f$ are in fact given by the Bernoulli numbers (I thank H.C. Herbig for this observation).
\end{remark}

From the previous lemma and Proposition \ref{Ghom} follows Theorem \ref{shearhomo}.


An example of application of the previous result is Corollary \ref{diffquot}:

\begin{proof}[Proof of Corollary \ref{diffquot}]
The assumption is equivalent to the statement that the curve $\{y=f(x)\}$ is homogeneous with respect to analytic maps of the kind $(x,y) \to (x+t, y+g_t(x))$, which are (particular) elements of ${\rm Shr}^\omega(I_t \times \mathbb R)$.
\end{proof}
It is interesting to observe that the conclusion of the corollary above does not hold if we only assume that $g_t$ is real-analytic for $t$ belonging to a dense subset of $I$: an example of such a function is constructed in \cite{De2}.

\

We can also check elementarily that the invariant curve property  holds (in a stronger form) for $\mathcal G = GL(2,\mathbb R)$:

\begin{proof}[Proof of Lemma \ref{linicp}]
We first notice that  the linear space $T_0(\gamma)$ is invariant under $A$, hence $A$ admits a real eigenvalue. The roots of its characteristic polynomial are thus both real, which means that $A$ can be conjugated to either a diagonal transformation $(x,y)\to (\lambda_1 x, \lambda_2 y)$ or to $(x,y)  \to (\lambda x + y, \lambda y)$ for some $\lambda_1,\lambda_2, \lambda \in \mathbb R\setminus\{0\}$. 

In the first case, we can assume up to permutation of the variables that $\gamma$ is defined by $\{y=f(x)\}$; moreover, if needed we can replace $A$ with $A^2$ or $A^{-2}$ (which do not give the identity since $A\not \in \mathcal R$) to achieve $0< \lambda_1\leq 1$ and $\lambda_2>0$. The invariance of $\gamma$ under $A$ corresponds to the equation $f(\lambda_1 x) = \lambda_2 f(x)$, holding for $x$ in a neighborhood of $0$. 
From this identity we see that, if $\lambda_1=1$, either $f\equiv 0$ or $\lambda_2=1$; since by hypothesis $A\neq Id$, we can assume $0<\lambda_1<1$. We can thus fix $k\in \mathbb N$ such that $\lambda_1^k < \lambda_2$. Differentiating the invariance equation $k$ times yields $f^{(k)}(\lambda_1 x) = (\lambda_2 / \lambda_1^k) f^{(k)}(x)$; by iteration, we also have $f^{(k)}(\lambda_1^j x) = (\lambda_2/\lambda_1^k)^j f^{(k)}(x)$ for all $j\in \mathbb N$. Since the $k$-th derivative of $f$ is bounded around $0$, this implies $f^{(k)}\equiv 0$, so that $f$ must be a polynomial of degree at most $k-1$.

We turn now to the case $A(x,y) = (\lambda x + y, \lambda y)$; as before we can assume $0<\lambda\leq 1$ (this entails passing to $A^2$ or $A^{-2}$ if needed and then conjugating the resulting linear transformation again to its Jordan form). The powers of $A$ can be directly computed as $A^j(x,y) = (\lambda^j x + j\lambda^{j-1}y, \lambda^j y)$ for all $j\in \mathbb Z$.
Since the $\lambda$-eigenspace of $A$ is generated by $\partial/\partial x$, we must have $\gamma = \{y = f(x)\}$ with $f(0)=f'(0)=0$. The local invariance equation is 
\begin{equation}\label{inveq}
f(\lambda x + f(x)) = \lambda f(x).
\end{equation}
As usual, we first suppose that $0<\lambda<1$. In this situation, considering the expression of the powers of $A$ it is easy to check that $A^{j}(p) \to 0$ as $j\to +\infty$ for all $p\in \mathbb R^2$, uniformly on compact subsets. It follows that, defining $\rho(x) = \lambda x + f(x)$ for $x\in \mathbb R$, $\rho^{\circ j}(x)\to 0$ (also uniformly on compact subsets of $\mathbb R$) as $j\to \infty$; the action of $\rho$ represents in fact the projection to $\mathbb R_x$ of the restriction of the action of $A$ to $\gamma$.

Differentiating (\ref{inveq}) two times, we obtain
\begin{equation}\label{prime}
f'(\lambda x + f(x) )\cdot (\lambda + f'(x)) = \lambda f'(x),
\end{equation}
\begin{equation}\label{doubleprime}
f''(\lambda x + f(x))\cdot (\lambda + f'(x))^2 + f'(\lambda x + f(x))f''(x) = \lambda f''(x).
\end{equation}
Note that, since $f'(0) = 0$, if $x$ is close enough to $0$ the term $\lambda + f'(x)$ does not vanish. We can thus solve for $f'(\lambda x + f(x))$ in (\ref{prime}) and replace it in (\ref{doubleprime}):
\[ f''(\lambda x + f(x))\cdot (\lambda + f'(x))^3 + \lambda f'(x)f''(x) = \lambda(\lambda + f'(x)) f''(x) \Rightarrow\]
\begin{equation}\label{replace}
\Rightarrow f''(\lambda x + f(x)) = f''(x) \frac{\lambda^2}{(\lambda + f'(x))^3}.
\end{equation}
Let $V$ be a neighborhood of $0$ in $\mathbb R$ such that $\frac{\lambda^2}{(\lambda + f'(x))^3} > C > 1$ for all $x\in V$; the choice of such a neighborhood is possible because, since $f'(0)=0$, $\frac{\lambda^2}{(\lambda + f'(x))^3} \to 1/\lambda > 1$ as $x\to 0$. From (\ref{replace}) we get that $|f''(\rho(x))| \geq C |f''(x)|$ for all $x\in V$. Since $\rho^{\circ j}\to 0$ uniformly, we can also choose a neighborhood $V'$ of $0$ such that $\rho^{\circ j}(V')\subset V$ for all $j\in \mathbb N$. Suppose now that there exists $x_0\in V'$ such that $f''(x_0)\neq 0$, and let $x_j = \rho^{\circ j}(x_0)$, $j\in \mathbb N$. Then $\{x_j\}_{j\in \mathbb N}\subset V$ and since $|f''(x_{j+1})| \geq C |f''(x_j)|$ for all $j\in \mathbb N$ we have $|f''(x_j)|\to \infty$, a contradiction. It follows that $f''\equiv 0$ on $V'$ (and since $f(0)=f'(0)=0$, we have in fact $f\equiv 0$).

Last, assume that $\lambda=1$, i.e.\ $A(x,y) = (x+y, y)$. The powers of $A$ are given by $A^j(x,y) = (x + j y, y)$ for all $j\in \mathbb Z$. Using (\ref{inveq}) with $\lambda = 1$ and iterating, we get that 
\begin{equation} \label{inveq1}
f(x + j f(x)) = f(x)
\end{equation}
 for all $j\in \mathbb Z$.  Supposing that $f\not\equiv 0$ in a neighborhood of $0$, we can choose a sequence $\{x'_j\}_{j\in \mathbb N}$ such that $x'_j\to 0$ and $f(x'_j)\neq 0$ for all $j\in \mathbb N$; since $f(0)=0$, we also have $f(x'_j)\to 0$ as $j\to \infty$. Fix now $x_0\in \mathbb R$ (close enough to $0$). We define a sequence $\{x_j\}_{j\in \mathbb N}$ as
\[ x_j = x'_j + \left \lfloor \frac{x_0 - x'_j}{f(x'_j)}\right \rfloor f(x'_j) \]
where $\lfloor \cdot \rfloor:\mathbb R\to \mathbb Z$ is the floor function $\lfloor x \rfloor = \max\{k\in \mathbb Z : k\leq x \}$. It is readily seen that $|x_0 - x_j|\leq |f(x'_j)|$ for all $j\in \mathbb N$; by construction, this implies $x_j\to x_0$ as $j\to \infty$. Moreover, by (\ref{inveq1}) follows $f(x_j) = f(x'_j)$ for any $j$. Hence $f(x_0) = \lim_{j\to \infty} f(x_j) = \lim_{j\to \infty} f(x'_j) = 0$, a contradiction. We conclude that, also in this case, $f\equiv 0$ in a neighborhood of $0$.
\end{proof}

\begin{remark}\label{cannot}
We note that, in this case, the regularity assumption on $\gamma$ cannot be weakened to $C^k$ for some fixed $k\in \mathbb N$. Indeed, if for example the linear map $A$ is diagonalizable with eigenvalues $0<\lambda_1,\lambda_2<1$ satisfying $\lambda_1^k\geq \lambda_2$, it is easy to construct invariant curves of class $C^k$ for $A$ which are not real-analytic outside $0$. However, given a linear map $A$ there exists $k_0\in \mathbb N$ such that every invariant curve for $A$ of class $C^{k_0}$ is real-analytic.
\end{remark}
As in the previous cases, we obtain as a consequence Theorem \ref{linhomo} (which, however, can be proved by more elementary means).


\section{Locally closed subsets}\label{lcs}
Let $K\subset \mathbb R^n$ be a locally closed subset; from \cite{RSS}, \cite{Sk} follows that if $K$ is $C^1$-homogeneous then it is a submanifold of class $C^1$. We wish to give a relatively simple argument showing that this result implies that the same is true for $C^r$-homogeneity, $r\in \mathbb N$ (and thus also for $C^\infty$-homogeneity), see Theorem \ref{Wilkin}.  Assuming Theorem \ref{Wilkin} for $r=1$, it is enough to prove the statement:

\emph{Any $C^1$ submanifold which is $C^r$-homogeneous is of class $C^r$.}
 
 \noindent The inductive procedure we use is analogous to the one employed in \cite{Wi}; however, our argument to prove the claim above relies on the results in \cite{RSS}, \cite{Sk} as the basis for the induction (while in \cite{Wi} that claim is proved in a self-contained way) and thus is somewhat simpler.

\subsection*{Proof of Theorem \ref{Wilkin} for $r>1$}

\

 Let $M\subset \mathbb R^n$ be an $m$-dimensional submanifold of $\mathbb R^n$ of class $C^1$. Up to a linear transformation, we can choose coordinates  $(x,y)$ for $\mathbb R^n \cong \mathbb R^m(x)\times \mathbb R^{n-m}(y)$, $x = (x_1,\ldots,x_m)$, $y=(y_1,\ldots,y_{n-m})$, such that $0\in M$ and $T_0(M)= \mathbb R^m(x)$. By the implicit function theorem, there exists a neighborhood $\mathcal N$ of $0$ in $\mathbb R^m$ and a map $f:\mathcal N\to \mathbb R^{n-m}$, of class $C^1$, such that $M\cap (\mathcal N\times \mathbb R^{n-m})=\{(x,y)\in \mathcal N\times \mathbb R^{n-m}: y=f(x)\}$. Since the arguments are all local, from now on we restrict without further mention to this neighborhood of $0$.

Let $Gr(m,\mathbb R^n)$ be the Grassmannian of $m$-dimensional subspaces of $\mathbb R^n$. A chart for $Gr(m,\mathbb R^n)$ in a neighborhood of $T_0(M)= \mathbb R^m(x)$ is given by the space of real $(n-m)\times m$ matrices; we denote by $\xi = (\xi_{j\ell})_{1\leq \ell\leq m}^{1\leq j \leq n-m}$ the local coordinates for this chart. For any point $x_0\in \mathbb R^m$, denote by $df(x_0)$ the $(n-m)\times m$ Jacobian matrix giving the differential of $f$ at $x_0$. Then the tangent space $T_p(M)$ of $M$ at $p=(x_0,f(x_0))$ is the $m$-dimensional linear subspace $\{y=df(x_0)x\}$, hence the coordinates of $T_p(M)$ in the chart for $Gr(m,\mathbb R^n)$ introduced above are given by the matrix $df(x_0)$.

We define the \emph{prolongation} $\mathcal M$ of $M$ as the subset of $\mathbb R^n \times Gr(m,\mathbb R^n)$ defined by $\mathcal M = \{(p,T_p M):p\in M\} \subset \mathbb R^n \times Gr(m,\mathbb R^n)$. 
By the previous paragraphs we have 
\[\mathcal M \cap (\mathcal N\times \mathbb R^{n-m}\times Gr(n,\mathbb R^n))  = \{(x,y,\xi)\in \mathcal N\times \mathbb R^{n-m}\times Gr(n,\mathbb R^n) : y = f(x), \xi=df(x)\}.\] 
Since $f$ is of class $C^1$, it follows that around $0$ $\mathcal M$ is the graph of the continuous map $(f,df):\mathbb R^m\to \mathbb R^{n-m}\times Gr(m,\mathbb R^n)$, therefore it is a locally closed subset of $\mathbb R^n\times Gr(m,\mathbb R^n)$. 

Let $\pi$ be the projection $\pi: \mathbb R^n \times Gr(m,\mathbb R^n) \to \mathbb R^m$, $\pi(x,y,\xi) = x$, and let $\pi|_{\mathcal M}$ be the restriction $\pi|_{\mathcal M}:\mathcal M \to \mathbb R^m$.

\begin{lemma}\label{shortrem}
The map $\pi|_{\mathcal M}:\mathcal M \to \mathbb R^m$ is a (local) homeomorphism.
\end{lemma}
\begin{proof}
Since $\pi: \mathbb R^n \times Gr(m,\mathbb R^n) \to \mathbb R^m$ is continuous, the restriction $\pi|_{\mathcal M}:\mathcal M \to \mathbb R^m$ is also continuous. Furthermore, by construction $\pi|_{\mathcal M}$ admits as inverse the continuous map $\mathbb R^m\to \mathcal M$ given by $x\to (x,f(x),df(x))$, defined around $0$.
\end{proof}

Suppose now that $M$ is a submanifold of class $C^{k+1}$ ($k\geq 1$). Then the map $f:\mathbb R^m\to \mathbb R^{n-m}$ is in turn of class $C^{k+1}$, and $(f,df):\mathbb R^m\to \mathbb R^{n-m}\times Gr(m,\mathbb R^n)$ is of class $C^k$: hence in this case $\mathcal M$ is a ($m$-dimensional) submanifold of class $C^k$.

The converse statement that $M$ is of class $C^{k+1}$ if $\mathcal M$ is $C^k$ is not true: $\mathcal M$ can be a real-analytic submanifold even if $M$ is not more than $C^1$-smooth (for example $M = \{y = x^{4/3}\}\subset \mathbb R^2 $, see the remark after Lemma 7.1 in \cite{Wi}). However, the converse is generically true in the following sense:

\begin{lemma}\label{Ckplus1} Suppose that $\mathcal M$ is a submanifold of class $C^k$, $k\geq 1$. Then there is a non-empty  open set $U\subset M$ such that $M\cap U$ is of class $C^{k+1}$.
\end{lemma}
\begin{proof}
 By Lemma \ref{shortrem} $\mathcal M$ must be $m$-dimensional, since $\pi|_{\mathcal M}:\mathcal M \to \mathbb R^m$ is a local homeomorphism. In what follows, to simplify the notation we write $\pi$ in place of $\pi|_{\mathcal M}$.
 
  Put $d=\min\{\dim\ker \pi|_{T_q{\mathcal M}}: q\in \mathcal M \}$, and let $W \subset \mathcal M$ be the set $W=\{q\in \mathcal M:\dim\ker \pi|_{T_q{\mathcal M}}=d\}$. Then $W$ is an open subset of $\mathcal M$ and the restriction of $\pi$ to $W$ has constant rank $m-d$. By the rank theorem, for any $q\in W$ the fiber $\pi^{-1}(\pi(q))$ of $\pi$ through $q$ is locally a manifold of dimension $d$. Should we have $d\geq 1$, this would contradict the fact that $\pi$ is a local homeomorphism: it follows that $d=0$.
  
Choose now $q_0\in W$, so that $\ker \pi|_{T_{q_0}{\mathcal M}}=\{0\}$, and let $x_0 = \pi(q_0)$. Since $\mathcal M$ is an $m$-dimensional manifold of class $C^k$, there exist a neighborhood $\mathcal V'$ of $q_0$ in $\mathbb R^n \times Gr(m,\mathbb R^n)$ and a vector-valued $C^k$ defining function $\rho:\mathcal V'\to \mathbb R^{n-m+(n-m)m}$ such that $\mathcal M\cap\mathcal V'=\{\rho=0\}$. The differential $d_{(y,\xi)}\rho(q_0)$ of $\rho$ with respect to the variables $(y,\xi)$ at $q_0$ is represented by a $(n-m+(n-m)m) \times (n-m+(n-m)m)$ square matrix which is invertible. Indeed, any element of the kernel of $d_{(y,\xi)}\rho(q_0)$ is a vector $v$  belonging to the subspace spanned by the variables $(y,\xi)$ such that $v\in T_{q_0}\mathcal M$. By definition of $\pi$ we also have $v\in\ker \pi|_{T_{q_0}{\mathcal M}}$, and since this space is trivial it follows that $v=0$.

Since $d_{(y,\xi)}\rho(q_0)$ is invertible, we can apply the implicit function theorem to obtain neighborhoods $V$ of $x_0$ in $\mathbb R^m$ and $\mathcal V$ of $q_0$ in $\mathbb R^n \times Gr(m,\mathbb R^n)$, and a map $F:V\to \mathbb R^{n-m} \times Gr(m,\mathbb R^n)$ of class $C^k$ such that $\mathcal M\cap \mathcal V=\{(y,\xi)=F(x)\}$. In other words, $\mathcal M$ can be locally expressed as the graph of a map $F$ of class $C^k$ defined around $x_0$. However we already know that $\mathcal M$ is also written as the graph of the map $\mathbb R^m\ni x\to (f(x),df(x))\in \mathbb R^{n-m} \times Gr(m,\mathbb R^n)$, i.e.\ $\mathcal M=\{(y,\xi)=(f(x),df(x))\}$. It follows that $F(x) = (f(x),df(x))$ for $x\in V$. Since $F$ is of class $C^k$, it follows in particular that $df$ is of class $C^k$ on $V$, which implies that $f$ is of class $C^{k+1}$ around $x_0$.
\end{proof}
\begin{remark} In fact, the open set $U$ in Lemma \ref{Ckplus1} is dense in $M$.
\end{remark}
To prove Theorem \ref{Wilkin}, we will now prove inductively the claim

\

\noindent {\bf P($m,n,k$):} An $m$-dimensional submanifold of $\mathbb R^n$ of class $C^1$ which is $C^k$-homogeneous is of class $C^k$.

\

Assume that {\bf P($m,n,k$)}  is true for all $k\leq k_0$ ($k_0\geq 1$), $m,n\in \mathbb N$, and let $M\subset \mathbb R^n$ be an $m$-dimensional $C^{k_0+1}$-homogeneous submanifold of class $C^1$. As observed above, the prolongation $\mathcal M\subset \mathbb R^n \times Gr(m,\mathbb R^n)$ is a locally closed subset. We will now check that 

\begin{lemma}
$\mathcal M$ is $C^{k_0}$-homogeneous.
\end{lemma}
\begin{proof}
Let $q_1,q_2\in \mathcal M$: then $q_1=(p_1,T_{p_1}M), q_2=(p_2,T_{p_2}M)$ for certain points $p_1,p_2\in M$. Since $M$ is $C^{k_0+1}$-homogeneous, there are neighborhoods $V_1,V_2$ of $p_1,p_2$ in $\mathbb R^n$ and a $C^{k_0+1}$-smooth diffeomorphism $\psi:V_1\to V_2$ such that $\psi(M\cap V_1)=M\cap V_2$ and $\psi(p_1)=p_2$.

We prolong now $\psi$ to a diffeomorphism $\widetilde\psi:V_1\times Gr(m,\mathbb R^n)\to V_2\times Gr(m,\mathbb R^n)$ by using the action of the differential $d\psi$ on the $m$-dimensional subspaces of $\mathbb R^n$. More precisely, write the components of $\psi^{-1}$ as $\psi^{-1}(x,y)=(g(x,y),h(x,y))$ (where $g:V_2\to \mathbb R^m(x)$ and $h:V_2\to \mathbb R^{n-m}(y)$ are maps of class $C^{k_0+1}$). At any point $(x,y)\in V_2$ the differentials $g_x(x,y),g_y(x,y),h_x(x,y),h_y(x,y)$ can be represented as matrices of dimension, respectively, $m\times m, m\times (n-m), (n-m)\times m$ and $(n-m)\times (n-m)$. We can write $\widetilde\psi:V_1\times Gr(m,\mathbb R^n)\to V_2\times Gr(m,\mathbb R^n)$ in the coordinates $(x,y,\xi)$ of our local chart as
\[\widetilde \psi(x,y,\xi)=\left(\psi(x,y), (h_y(\psi(x,y))  - \xi g_y(\psi(x,y)))^{-1}(\xi g_x(\psi(x,y)) - h_x(\psi(x,y))) \right);\]
here the map $(h_y  - \xi g_y)^{-1}(\xi g_x - h_x)$ acts rationally on the local chart of  $Gr(m,\mathbb R^n)$ (and extends as an algebraic diffeomorphism $Gr(m,\mathbb R^n)\to Gr(m,\mathbb R^n)$) for any fixed $(x,y)\in V_2$, and it depends $C^{k_0}$-smoothly on $(x,y)$. It follows that $\widetilde\psi:V_1\times Gr(m,\mathbb R^n)\to V_2\times Gr(m,\mathbb R^n)$ is a diffeomorphism of class $C^{k_0}$.

Since $\psi$ maps $M$ into $M$, it follows that for any $p\in M\cap V_1$ the differential of $\psi$ at $p$ sends $T_pM\in Gr(m,\mathbb R^n)$ to $T_{\psi(p)}M\in Gr(m,\mathbb R^n)$, so that $\widetilde\psi(p,T_pM)=(\psi(p),T_{\psi(p)}M)$. This shows that $\widetilde \psi(\mathcal M \cap (V_1\times Gr(m,\mathbb R^n)))\subset \mathcal M \cap (V_2\times Gr(m,\mathbb R^n))$, and $\widetilde \psi(q_1)=q_2$. On the other hand let $q\in \mathcal M \cap (V_2\times Gr(m,\mathbb R^n))$, $q=(p,T_pM)$ with $p\in M\cap V_2$. Then $p'=\psi^{-1}(p)\in M\cap V_1$, $q'=(p',T_{p'}M)\in \mathcal M \cap (V_1\times Gr(m,\mathbb R^n))$ and $\widetilde\psi(q') =(\psi(p'),T_{\psi(p')}M)=(p,T_pM)=q$. This shows that in fact $\widetilde \psi(\mathcal M \cap (V_1\times Gr(m,\mathbb R^n)))= \mathcal M \cap (V_2\times Gr(m,\mathbb R^n))$.

Since $q_1,q_2$ were arbitrary points of $\mathcal M$, we conclude that $\mathcal M$ is $C^{k_0}$-homogeneous.
\end{proof}
By the previous lemma, since $\mathcal M$ is locally closed, applying \cite{RSS}, \cite{Sk} we have that $\mathcal M$ is a submanifold of class $C^1$. The inductive assumption {\bf P($m,n + m(n-m),k_0$)} then implies that $\mathcal M$ is of class $C^{k_0}$. By Lemma \ref{Ckplus1} there is an open set $U\neq \emptyset$ such that $M\cap U$ is of class $C^{k_0+1}$: by homogeneity, then, $M$ is of class $C^{k_0 + 1}$ everywhere, which proves {\bf P($m,n,k_0+1$)} and concludes the proof of Theorem \ref{Wilkin}.

\

As a corollary of Theorem \ref{Wilkin} and of Proposition \ref{Ghom}, we immediately get Theorem \ref{disopcur}. In particular, the result applies to the subgroups considered in section \ref{aninvcur}, i.e.\ linear maps, analytic shears and holomorphic transformations.


\

\

\noindent {\bf Acknowledgments.}\  The author is very grateful to Bernhard Lamel for many discussions related to the topic of this paper; in particular, among other things, for pointing to Cartan's work \cite{Ca}.


\end{document}